\documentclass[reqno]{amsart}
\usepackage{amssymb,amsmath,amsthm,mathrsfs,paralist,cite}
\RequirePackage[colorlinks,citecolor=blue,urlcolor=blue]{hyperref}

\DeclareMathOperator{\Var}{Var}

\newcommand{\N}{\mathbb{N}}

\newcommand{\R}{\mathbb{R}}

\renewcommand{\P}{\mathbb{P}}
\newcommand{\E}{\mathbb{E}}
\renewcommand{\d}{{\rm d}}
\renewcommand{\Re}{\mathrm{Re}}

\author[C.Y. Lee and T. Zhang]{Cheuk Yin Lee \and Tianyi Zhang}
\address{School of Science and Engineering, The Chinese University of Hong Kong (Shenzhen), Longgang, Shenzhen, Guangdong 518172, China}
\email{leecheukyin@cuhk.edu.cn}
\address{School of Science and Engineering, The Chinese University of Hong Kong (Shenzhen), Longgang, Shenzhen, Guangdong 518172, China}
\email{tianyizhang1@link.cuhk.edu.cn}

\title[Sharp moduli of continuity]{Sharp moduli of continuity for Gaussian fields and stochastic PDEs via correlation bounds}

\newtheorem{theorem}{Theorem}[section]

\newtheorem{assumption}[theorem]{Assumption}
\newtheorem{lemma}[theorem]{Lemma}

\newtheorem{proposition}[theorem]{Proposition}

\newtheorem{remark}[theorem]{Remark}

\newtheorem{example}[theorem]{Example}
\numberwithin{equation}{section}


\begin{document}

\begin{abstract}
    Exact uniform and local moduli of continuity for anisotropic Gaussian random fields are established under a general framework based on correlation bounds for pairwise increments. This framework does not necessarily require strong local nondeterminism (SLND), stationarity of increments, or spectral-type representations.
    As an application, we solve an open problem about sharp moduli of continuity for a class of linear stochastic PDEs in a particular case where the solution is $C^{1-}$ in space and SLND is not available.
    In this case, we establish decorrelation of the spatial increments via new localization estimates, which may be of independent interest.
    We also briefly discuss an example about Gaussian Volterra processes.
\end{abstract}
\maketitle

\setcounter{tocdepth}{3}
\let\oldtocsection=\tocsection
\let\oldtocsubsection=\tocsubsection
\let\oldtocsubsubsection=\tocsubsubsection

\renewcommand{\tocsection}[2]{\hspace{0em}\oldtocsection{#1}{#2}}
\renewcommand{\tocsubsection}[2]{\hspace{2.5em}\oldtocsubsection{#1}{#2}}

\section{Introduction}

Continuity and regularity properties are fundamental in the studies of stochastic processes.
For the standard Brownian motion $\{B(t)\}_{t \ge 0}$, L\'evy's uniform modulus of continuity (see, e.g., \cite[Theorem I.2.7]{RY}) states that
\begin{align*}
    \lim_{h \to 0^+} \sup_{t, s \in [0,1]: 0<|t-s| \le h} \frac{|B(t)-B(s)|}{\sqrt{|t-s| \log(1/|t-s|)}} = \sqrt{2} \quad \text{a.s.}
\end{align*}
The local modulus of continuity is given by the law of the iterated logarithm (see, e.g., \cite[Theorem II.1.9]{RY}): For each fixed $t_0 \ge 0$,
\begin{align*}
    \lim_{h \to 0^+} \sup_{t: 0<|t-t_0| \le h} \frac{|B(t)-B(t_0)|}{\sqrt{|t-t_0| \log(1/|t-t_0|)}} = \sqrt{2} \quad \text{a.s.}
\end{align*}
If $\{X(t)\}_{t \ge 0}$ is a Gaussian process with $\E[(X(t)-X(s))^2]^{1/2} \asymp \varphi(t-s)$ for $0\le s\le t\le \varepsilon\le 1$ and $\varphi$ is non-decreasing, it is a well-known fact \cite{Fernique, Dudley, MS70, MS72} that $X$ is a.s.~continuous on $[0\,,\varepsilon]$ if and only if
\begin{align}\label{Fernique}
    \int_0^\varepsilon \frac{\varphi(r) \, \d r}{r \sqrt{\log(1/r)}} < \infty.
\end{align}
Further results for uniform and local moduli of continuity for Gaussian processes and Gaussian random fields can be found in \cite{OP73, Nisio, M68, M70, MR, MR92, Kono, SW70, MS72, MWX13, LWX15, LMX18, LX23, TTV04, H23}.
In particular, some of the sharp results in \cite{Kono, M68, M70, Nisio, MR} require concavity of $\varphi^2$ or stationarity of increments; \cite{MWX13, LX23} have developed a framework regardless of concavity and based on strong local nondeterminism to establish, among other things, exact moduli of continuity for anisotropic Gaussian random fields including linear stochastic partial differential equations (SPDEs) driven additive Gaussian noise.
See, e.g., \cite{AX17, TX17, LanX19, HSWX20, LX19, L22, HL, GSWX25, CLX, TTV04} for further results related to SPDEs.

Recall that a centered Gaussian random field $\{X(t)\}_{t \in I \subset \R^N}$ is said to satisfy \emph{local nondeterminism} (LND) if for every $n \in \N$, there exists $C_n>0$ such that
\begin{align}\label{LND}
    \Var(X(t)|X(t_1),\dots, X(t_n)) \ge C_n \min_{1\le i \le n}\Var(X(t)-X(t_i)) 
\end{align}
uniformly for all $t,t_1,\dots, t_n \in I$; it is said to satisfy \emph{strong local nondeterminism} (SLND) if the constant $C_n=C_0>0$ does not depend on $n$ \cite{Berman, Pitt, MP87, CD82}.
The SLND property makes it possible to study sample path properties of many Gaussian processes without independent increments or Markov property.
See \cite{X08, X09} and the references therein for more information, and \cite{LX23, LX, L22H, KLPX, LL, LT, KL} for some recent developments and applications of SLND.

However, SLND may not be easy to prove for certain Gaussian processes and SPDEs.
The goal of this paper is to develop a framework without necessarily requiring SLND to establish exact uniform and local moduli of continuity for general Gaussian random fields and SPDEs.

Throughout, $X=\{X(t)\}_{t \in \R^N}$ denotes a real-valued centered Gaussian random field defined on a complete probability space $(\Omega\,, \mathcal{F}\,,\P)$ with continuous sample functions.
Let $I$ be a fixed compact $N$-dimensional interval in $\R^N$.
We have the following result for exact uniform modulus of continuity:

\begin{theorem}\label{th:um}
    Suppose $\E[(X(t)-X(s))^2]^{1/2} \asymp \varphi(|t-s|)$ for all $t,s \in I$, where $\varphi$ is non-decreasing, continuous and regularly varying at 0 of index $\theta \in [0\,,1]$ with
    \begin{align}\label{um:phi}
        \int_0^h \frac{\varphi(r) \, \d r}{r \sqrt{\log(1/r)}} \lesssim \varphi(h) \sqrt{\log(1/h)} = o(1) \quad \text{as $h \to 0^+$.}
    \end{align}
    If there exist $C_0>0$, $\rho >0$, $a_0>0$, and $h_0 \in (0\,,1]$ such that
    \begin{align}\label{um:cv}
        \Var(X(t)-X(s)|X(t')-X(s')) \ge C_0 \Var(X(t)-X(s))
    \end{align}
    uniformly for all $h \in (0\,,h_0]$ and for all $t,s,t',s' \in I$ with $|t-s| \le h$, $|t'-s'| \le h$ and $|s-s'| \ge a_0 h^\rho$, then
    \begin{align}\label{um}
        \lim_{h\to0^+}\sup_{t,s \in I: 0<|t-s| \le h} \frac{|X(t)-X(s)|}{\varphi(|t-s|) \sqrt{\log(1/|t-s|)}} = K_1 \quad \text{a.s.}
    \end{align}
    where $0<K_1<\infty$ is a constant.
\end{theorem}

The next result yields exact local modulus of continuity:

\begin{theorem}\label{th:lm}
    Fix $t_0 \in I$. Suppose $\E[(X(t)-X(s))^2]^{1/2} \asymp \varphi(|t-s|)$ for all $t,s \in I$, where $\varphi$ is non-decreasing, continuous and regularly varying at 0 of index $\theta \in [0\,,1]$ with
    \begin{align}\label{lm:phi}
        \int_0^1 \frac{\varphi(hr) \, \d r}{r\sqrt{\log(2/r)}} \lesssim \varphi(h)\sqrt{\log\log(1/h)}=o(1) \quad \text{as $h\to0^+$.}
    \end{align}
    If there is a sequence $(t_n)_{n=n_0}^\infty$ in $I$ with $|t_n-t_0| = a^{-n}$ for some $a>1$, and there exists $C_0>0$ such that
    \begin{align}\label{lm:cv}
        \Var(X(t_n)-X(t_0)|X(t_m)-X(t_0))
        \ge C_0 \Var(X(t_n)-X(t_0))
    \end{align}
    uniformly for all $m>n\ge n_0$, then
    \begin{align}\label{lm}
        \lim_{h \to 0^+} \sup_{t \in I: 0<|t-t_0|\le h} \frac{|X(t)-X(t_0)|}{\varphi(|t-t_0|) \sqrt{\log\log(1/|t-t_0|)}} = K_2 \quad \text{a.s.}
    \end{align}
    where $0<K_2<\infty$ is a non-random number possibly depending on $t_0$.
\end{theorem}



\begin{remark}
    Both conditions \eqref{um:cv} and \eqref{lm:cv}, which involve only conditional variance of pairwise increments, can be deduced from LND \eqref{LND} (see Proposition \ref{prop:SLND}) and are equivalent to correlation bounds for the increments according to Lemma \ref{lem:cv:cov} below.
    Both conditions allow to apply Slepian's lemma to prove sharpness (lower bound) of the moduli of continuity.
\end{remark}

\begin{remark}
    If $\varphi(h)\sqrt{\log(1/h)}$ (resp.~$\varphi(h)\sqrt{\log\log(1/h)}$) is increasing near 0, then Theorem \ref{th:um} (resp.~Theorem \ref{th:lm}) also yields the exact m-modulus of continuity
    \[
        \limsup_{h\to0^+} \sup_{t,s\in I: 0<|t-s|\le h}\frac{|X(t)-X(s)|}{\varphi(h)\sqrt{\log(1/h)}} = K_1
    \]
    (resp.
    \[
        \limsup_{h\to0^+} \sup_{t\in I: 0<|t-t_0|\le h} \frac{|X(t)-X(t_0)|}{\varphi(h)\sqrt{\log\log(1/h)}} = K_2\, );
    \]
    see Lemma \ref{lem:m-mod} below.
\end{remark}

Theorems \ref{th:um} and \ref{th:lm} are special cases of more general results given by Theorems \ref{th:um:aniso} and \ref{th:lm:aniso} below for anisotropic Gaussian random fields, which also yield upper and lower bounds for the constants $K_1$ and $K_2$ above.
Both Theorems \ref{th:um} and \ref{th:lm} are applicable under the SLND assumption (see Proposition \ref{prop:SLND}).
These results do not necessarily require SLND, stationarity of increments, or spectral-type representations, and therefore provide a more general framework compared to \cite{MWX13, LX23}.

Furthermore, our results enable us to solve a particular case of an open problem about the exact uniform moduli of continuity for a class of linear SPDEs.
Specifically, let $\alpha \in (0\,,2]$, $H \in (1/2\,,1)$, and consider the SPDE
\begin{align}\label{SPDE}
    \partial_t u(t\,,x) = -(-\Delta)^{\alpha/2} u(t\,,x) + \dot{W}(t\,,x) \quad t \ge 0, x \in \R,
\end{align}
with initial condition $u(0\,,\cdot) = 0$, where $-(-\Delta)^{\alpha/2}$ denotes fractional Laplacian and $\dot{W}$ is a centered generalized Gaussian noise that is fractional in time and white in space, with covariance
\begin{align*}
    \E[\dot{W}(t\,,x) \dot{W}(s\,,y)] = |t-s|^{2H-2} \delta(x-y).
\end{align*}
A pointwise defined mild solution to \eqref{SPDE} exists if and only if $2H\alpha > 1$, and the exact moduli of continuity for $u$ is known when $1<2H\alpha<3$ (see \cite{LX23, HSWX20} and Section \ref{s:SPDE} below under a more general setting).
When $2H\alpha = 3$, it is known that for a fixed compact rectangle $I\times J \subset (0\,,\infty) \times \R$,
\[
    \E[(u(t\,,x)-u(s\,,y))^2]^{1/2} \asymp \phi((t\,,x)\,,(s\,,y))
\]
uniformly for all $(t\,,x), (s\,,y) \in I \times J$, where
\[
    \phi((t\,,x)\,,(s\,,y)) := |t-s|^{1/\alpha} \vee |x-y|\sqrt{\log_+( 1/|x-y|)}.
\]
Throughout, we write $a \vee b := \max\{a\,,b\}$ and $\log_+(a) := \log(a \vee e)$.
In this case, the Kolmogorov continuity theorem implies that $u\in C^{1/\alpha-, 1-}$~a.s.
As was discussed in \cite{H23, HSWX20, CLX}, when $2H\alpha = 3$, whether the solution satisfies SLND in $x$ is an open problem, and thus the exact uniform modulus of continuity in $x$ is also open.
Also, the modulus of continuity results in \cite{M70, M68, MR, Kono} cannot be applied because the function $h^2\log_+(1/h) \asymp \E[(u(t\,,x+h)-u(t\,,x))^2]$ is not concave.
Despite the absence of SLND, we are able to apply our results to solve the exact uniform modulus of continuity:

\begin{theorem}\label{th:um:SPDE}
Suppose $2H\alpha =3$. Fix a compact rectangle $I\times J \subset (0\,,\infty) \times \R$, $t_0 > 0$ and $x_0 \in \R$.
Then, a.s.,
\begin{align}\label{um:SPDE:x}
    \lim_{h\to0^+} \sup_{x,y \in J: 0<|x-y|\le h} \frac{|u(t_0\,,x)-u(t_0\,,y)|}{|x-y|\log(1/|x-y|)} = K_3,
\end{align}
\begin{align}\label{um:SPDE:t}
    \lim_{h\to0^+} \sup_{t,s\in I: 0<|t-s| \le h} \frac{|u(t\,,x_0)-u(s\,,x_0)|}{|t-s|^{1/\alpha}\sqrt{\log(1/|t-s|)}} = K_4,
\end{align}
\begin{align}\label{um:SPDE:tx}
    \lim_{h\to0^+} \sup_{z,z'\in I\times J: 0<\phi(z,z') \le h} \frac{|u(z)-u(z')|}{\phi(z\,,z')\sqrt{\log(1/\phi(z\,,z'))}} = K_5,
\end{align}
for some non-random numbers $K_3, K_4, K_5 \in (0\,,\infty)$ that may depend on $I, J$ or $t_0$.
\end{theorem}

In particular, \eqref{um:SPDE:x} is a consequence of Theorem \ref{th:um}. The proof of \eqref{um:SPDE:x} is based on new localization estimates for the spatial increments of the solution (Proposition \ref{prop:loc}), which extend the idea of \cite{FKM15} and yield decorrelation of the spatial increments (Proposition \ref{prop:corr}), hence verifying condition \eqref{um:cv}.
The localization estimates and decorrelation result may be of independent interest.
Besides, we also obtain the exact local modulus of continuity:

\begin{theorem}\label{th:lm:SPDE}
Suppose $2H\alpha =3$, and fix $z_0 = (t_0\,,x_0) \in (0\,,\infty) \times \R$. Then, a.s.,
\begin{align}\label{lm:SPDE:x}
    \lim_{h\to0^+} \sup_{x: 0<|x-x_0|\le h} \frac{|u(t_0\,,x)-u(t_0\,,x_0)|}{|x-x_0|\sqrt{\log(1/|x-x_0|) \log\log\log(1/|x-x_0|)}} = K_6,
\end{align}
\begin{align}\label{lm:SPDE:t}
    \lim_{h\to0^+} \sup_{t: 0<|t-t_0|\le h} \frac{|u(t\,,x_0)-u(t_0\,,x_0)|}{|t-t_0|^{1/\alpha}\sqrt{\log\log(1/|t-t_0|)}} = K_7,
\end{align}
\begin{align}\label{lm:SPDE:tx}
    \lim_{h\to0^+} \sup_{z : 0<\phi(z,z_0) \le h} \frac{|u(z)-u(z_0)|}{\phi(z\,,z_0)\sqrt{\log\log(1/\phi(z\,,z_0))}} = K_8,
\end{align}
for some non-random numbers $K_6,K_7,K_8 \in (0\,,\infty)$ that may depend on $t_0$.
\end{theorem}

Let us point out that the local modulus of continuity in \eqref{lm:SPDE:x} cannot be obtained from Theorem \ref{th:lm} or the general result Theorem \ref{th:lm:aniso} below.
This is because the spatial increments are highly correlated. The proof of \eqref{lm:SPDE:x} uses the fact that the much faster rate $\exp(-e^{n^\gamma})$ for the convergence $|x_n - x_0|\to 0$ is required in order for the increments to have sufficient decorrelation in the sense of \eqref{lm:cv}.
The proof also relies on error estimates for the linearization of spatial increments involving an approximate derivative.
To the best of our knowledge, \eqref{lm:SPDE:x} is a new result.

The paper is organized as follows. 
In Section \ref{s:general:results}, we present our general results Theorems \ref{th:um:aniso} and \ref{th:lm:aniso}.
In Section \ref{s:SLND}, we compare our conditions with the SLND condition.
Section \ref{s:pre} contains some preliminary materials concerning exact moduli of continuity, zero-one laws, and Slepian's lemma.  
In Section \ref{s:pf:general}, we prove Theorems \ref{th:um:aniso} and \ref{th:lm:aniso}.
In Section \ref{s:pf:main}, we prove Theorems \ref{th:um} and \ref{th:lm}.
In Section \ref{s:SPDE}, we consider the example about linear SPDEs, establish localization results, and prove Theorems \ref{th:um:SPDE} and \ref{th:lm:SPDE}.
Finally, in Section \ref{s:volterra}, we discuss another example about a class of Gaussian Volterra processes whose moduli of continuity may or may not be H\"older continuous.

The following list of notation will be used  throughout the paper: $\N = \{1 ,2 ,\dots\}$; ${\bf 1}_A$ denotes indicator function of a set $A$; $a\wedge b = \min\{a\,,b\}$; $a\vee b = \max\{a\,,b\}$; $\log_+(a) = \log(a \vee e)$; ``$f(x) \lesssim g(x)$'' means that there exists $C \in (0\,,\infty)$ such that $f(x) \le C g(x)$ for all $x$; ``$f(x) \asymp g(x)$'' means that $f(x) \lesssim g(x)$ and $g(x) \lesssim f(x)$; ``$f(x) \sim g(x)$ as $x\to a$'' means that $f(x)/g(x) \to 1$ as $x\to a$; ``$f(x) \propto g(x)$'' means that there exists $C\in (0\,,\infty)$ such that $f(x) = Cg(x)$ for all $x$;
for any $t \in \R^N$, we write $t = (t^1,\dots, t^N)$ and $|t|$ as its Euclidean norm; and $\|Y\|_2 = (\E|Y|^2)^{1/2}$ for any random variable $Y$.

\section{General results}
\label{s:general:results}

In this section, we present general uniform and local moduli of continuity results for anisotropic Gaussian random fields.
Throughout, we write
\begin{align}\label{sigma}
    \sigma(t\,,s) = \|X(t)-X(s)\|_2 \quad \forall t, s \in I.
\end{align}

\begin{assumption}\label{a:sigma:aniso}
    There exist $N$ functions $\phi_1,\dots, \phi_N: [0\,,\infty) \to [0\,,\infty)$ with $\phi_i(0)=\phi_i(0+)=0$ that are non-decreasing, continuous, and regularly varying at 0 of the form $\phi_i(r) = r^{\theta_i} L_i(r)$, where $\theta_i \in [0\,,1]$, $L_i: (0\,,\infty) \to (0\,,\infty)$ is slowly varying at 0, and there exist two constants $C_1, C_2 > 0$ such that
    \begin{align}\label{sigma:phi}
        C_1 \phi(t-s) \le \sigma(t\,,s) \le C_2 \phi(t-s) \quad \forall t,s \in I,
    \end{align}
    where $\phi(t-s) := \max_{1\le i \le N} \phi_i(|t^i-s^i|)$.
\end{assumption}

Throughout, $\psi: [0\,,\infty) \to [0\,,\infty)$ denotes the function defined by
\begin{align}\label{psi}
    \psi(r) = \prod_{i=1}^N \phi_i^{-1}(r) \quad \forall r \ge 0,
\end{align}
where $\phi_i^{-1}$ denotes the generalized inverse $\phi_i^{-1}(r) = \inf\{ h \ge 0: \phi(h) \ge r \}$.

\begin{assumption}\label{a:psi:aniso}
    The function $\psi$ satisfies
        \begin{gather}
            \lim_{h \to 0^+}\psi^{-1}(h) \sqrt{\log(1/h)} = 0 \label{psi:1} \\
            \text{and} \quad C_3:= \limsup_{h \to 0^+} \frac{1}{\psi^{-1}(h)\sqrt{\log(1/h)}} \int_0^h \frac{\psi^{-1}(r)\, \d r}{r\sqrt{\log(1/r)}} < \infty.\label{psi:2}
        \end{gather}
\end{assumption}

\begin{assumption}\label{a:cv:aniso}
    There exist constants $C_0 \in (0\,,1]$, $A_0>0$, $\rho >0$ and $h_0 \in (0\,,1]$ such that for every $h \in (0\,,h_0]$,
    \begin{align*}
        \E[(X(t)-X(s))(X(t')-X(s'))] \le \sqrt{1-C_0} \|X(t)-X(s)\|_2 \|X(t')-X(s')\|_2
    \end{align*}
    for all $t,s,t',s' \in I$ with $\phi(t-s)\le h$, $\phi(t'-s') \le h$ and $\phi(s-s') \ge \psi^{-1}(A_0 (\psi(h))^\rho)$, where $\psi$ is defined by \eqref{psi}.
\end{assumption}

\begin{theorem}\label{th:um:aniso}
    Under Assumptions \ref{a:sigma:aniso} and \ref{a:psi:aniso},
    \begin{align}\label{um:aniso}
        \lim_{h \to 0^+} \sup_{t, s \in I: 0<\phi(t-s) \le h} \frac{|X(t)-X(s)|}{\phi(t-s) \sqrt{\log[1/\psi(\phi(t-s))]}} = C \quad \text{a.s.}
    \end{align}
    for some constant $C \le K C_2(1+C_3/2)<\infty$, where $K$ is a universal constant.
    If in addition Assumption  \ref{a:cv:aniso} holds, then
    we also have $C \ge C_1[2\rho (1-\sqrt{1-C_0})]^{1/2}>0$.
\end{theorem}


\begin{assumption}\label{a:psi:lil}
    There exist $a_0>0$ and an increasing function $\ell:[a_0\,,\infty) \to (0\,,\infty)$ with $\lim_{x\to\infty}\ell(x)=\infty$ such that
    \begin{gather}\label{psi:lil:1}
        \lim_{h\to0^+}\psi^{-1}(h)\sqrt{\ell(1/h)} =0\\
        \text{and} \quad 
        C_4:=\limsup_{h\to0^+} \frac{1}{\psi^{-1}(h) \sqrt{\ell(1/h)}} \int_0^{1} \frac{\psi^{-1}(h r) \, \d r}{r \sqrt{\log(2/r)}} < \infty.\label{psi:lil:2}
    \end{gather}
\end{assumption}

\begin{assumption}\label{a:cv:lil}
    There exists a sequence $(t_n)_{n=1}^\infty$ in $I$ with 
    $\psi(\phi(t_n-t_0)) \asymp h_n$,  $h_n \downarrow 0$,
    and there exist $C_0\in (0\,,1]$, $n_0 \in \N$ such that
    \begin{align*}
        \E[(X(t_n) - X(t_0))(X(t_m)-X(t_0))] \le \sqrt{1-C_0} \|X(t_n) - X(t_0)\|_2 \|X(t_m)-X(t_0)\|_2
    \end{align*}
    uniformly for all $m > n \ge n_0$.
\end{assumption}

\begin{theorem}\label{th:lm:aniso}
    Fix $t_0 \in I$.
    Suppose either
    \begin{enumerate}[(1)]
        \item $\ell(x) = \log\log(x)$ and $h_n = e^{-\rho n^\gamma}$, where $\rho>0$ and $\gamma>0$; or
        \item $\ell(x) = \log(x)$ and $h_n = n^{-\gamma}$, where $\gamma>0$.
    \end{enumerate}
    Then, under Assumptions \ref{a:sigma:aniso} and \ref{a:psi:lil}, we have
    \begin{align}\label{lm:aniso}
        \lim_{h\to0^+} \sup_{t\in I: 0<\phi(t-t_0) \le h} \frac{|X(t)-X(t_0)|}{\phi(t-t_0)\sqrt{\ell(1/ \psi(\phi(t-t_0)))}} = C \quad \text{a.s.}
    \end{align}
    for some non-random number $C=C(t_0)$ such that
    \begin{align}\label{lm:C:H<1}
        C \le \begin{cases}
        (\sqrt{2}+KC_4/2)C_2 & \text{in Case (1),}\\
        (\sqrt{2/\nu}+KC_4/2)C_2
        & \text{in Case (2),}
    \end{cases}
    \end{align}
    where $K$ is a universal constant and $\nu = \sum_{i=1}^N \theta_i^{-1}$, with $1/0=\infty$ and $1/\infty = 0$.
    If in addition Assumption \ref{a:cv:lil} holds,
    then $C \ge C_1 [2\gamma^{-1}(1-\sqrt{1-C_0})]^{1/2} > 0$ in both Cases (1) and (2).
\end{theorem}

\begin{remark}\label{rmk:theta>0}
    If $\theta_i > 0$ for all $1 \le i \le N$, then \eqref{psi:1} and \eqref{psi:lil:1} automatically hold in both Cases (1) and (2).
\end{remark}

\section{Comparison with SLND}
\label{s:SLND}

We show that conditions \eqref{um:cv} and \eqref{lm:cv} both hold under the LND condition in \eqref{LND}. 
In particular, they both hold under the framework of SLND \cite{LX23} (or under one-sided SLND in Lemma \ref{lem:volterra}(ii) below).

\begin{proposition}\label{prop:SLND}
    Suppose $\Var(X(t)-X(s)) \asymp \varphi^2(|t-s|)$ for all $t,s \in I$, where $\varphi$ is non-decreasing, continuous and regularly varying at 0 of index $\theta \in [0\,,1]$.
    If X satisfies LND \eqref{LND} for $n=3$, then \eqref{um:cv} holds with $\rho=1$ and $a_0=4$.
    If X satisfies LND \eqref{LND} for $n=2$, then \eqref{lm:cv} holds for any $a>1$.
\end{proposition}

\begin{proof}
    First, if $|t-s|\le h$, $|t'-s'|\le h$ and $|s-s'| \ge 4 h$, then by triangle inequality, 
    \begin{align*}
        |t-t'| \ge 2h \quad \text{and} \quad 
        |t-s'| \ge 3h.
    \end{align*}
    Hence, a property of conditional variance, LND, and monotonicity of $\varphi$ imply that
    \begin{align*}
        &\Var(X(t)-X(s)|X(t')-X(s'))
        \ge \Var(X(t)|X(s),X(t'),X(s'))\\
        &\gtrsim \Var(X(t)-X(s))\wedge \Var(X(t)-X(t'))\wedge \Var(X(t)-X(s'))\\
        &\gtrsim \varphi^2(|t-s|)\wedge \varphi^2(|t-t'|)\wedge \varphi^2(|t-s'|)
        = \varphi^2(|t-s|) \gtrsim \Var(X(t)-X(s)).
    \end{align*}
    This verifies \eqref{um:cv} with $\rho=1$.
    Next, we verify \eqref{lm:cv}. For any fixed $a>1$, we may choose $(t_n)_{n=1}^\infty$ in $I$ converging to $t_0$ in a linear direction such that $|t_n-t_0|=a^{-n}$ and $|t_n-t_m| = a^{-n}-a^{-m}$ for all $m>n\ge 1$.
    Then for $m>n \ge 1$,
    \[
        |t_n-t_m| \ge |t_n-t_{n+1}| = a^{-n}-a^{-n-1} = (1-a^{-1}) a^{-n}.
    \]
    Since $\varphi$ is non-decreasing and regularly varying (see \eqref{rv} below), the preceding together with LND yields
    \begin{align*}
        &\Var(X(t_n)-X(t_0)|X(t_m)-X(t_0)) \ge \Var(X(t_n)|X(t_m),X(t_0))\\
        & \gtrsim \Var(X(t_n)-X(t_m)) \wedge \Var(X(t_n)-X(t_0)) \\
        &\gtrsim \varphi^2(|t_n-t_m|) \wedge \varphi^2(|t_n-t_0|) = \varphi^2(|t_n-t_m|) \ge \varphi^2((1-a^{-1})a^{-n})\\
        & \gtrsim \varphi^2(a^{-n}) = \varphi^2(|t_n-t_0|) \gtrsim \Var(X(t_n)-X(t_0))
    \end{align*}
    uniformly for all sufficiently large $m>n \gg 1$. This verifies \eqref{lm:cv}.
\end{proof}

\section{Preliminaries}
\label{s:pre}

Let us recall the definition of regularly varying functions.
Let $\theta \ge 0$ and $r_0 > 0$.
A function $\varphi: [0\,,r_0] \to [0\,,\infty)$ is said to be \emph{regularly varying} at 0 of index $\theta$ if
\begin{align}\label{rv}
    \lim_{r\to0^+} \frac{\varphi(\lambda r)}{\varphi(r)} = \lambda^{\theta} \quad \forall \lambda>0;
\end{align}
it is said to be \emph{slowly varying} if it is regularly varying of index 0.
In particular, it follows from definition that if $\varphi$ is regularly varying of index $\theta$, then for any $\lambda>0$ and $\delta >0$, there exists $r_1\in (0\,,r_0]$ such that for all $r \in (0\,,r_1]$,
\begin{align}\label{rv:bd}
    \begin{cases}
        \lambda^{\theta+\delta} \varphi(r) \le \varphi(\lambda r) \le \lambda^{\theta-\delta} \varphi(r) & \text{if $0<\lambda<1$,}\\
        \lambda^{\theta-\delta} \varphi(r) \le \varphi(\lambda r) \le \lambda^{\theta+\delta} \varphi(r) & \text{if $\lambda > 1$.}
    \end{cases}
\end{align}
Furthermore, if $\varphi$ is non-decreasing, so is $\varphi^{-1}$, and hence for all $r \in (0\,,r_1]$,
\begin{align}\label{rv:bd:inv}
    \begin{cases}
        \varphi^{-1}(\lambda^{\theta+\delta} \varphi(r)) \le \lambda r \le \varphi^{-1}(\lambda^{\theta-\delta} \varphi(r)) & \text{if $0<\lambda<1$,}\\
        \varphi^{-1}(\lambda^{\theta-\delta} \varphi(r)) \le \lambda r \le \varphi^{-1}(\lambda^{\theta+\delta} \varphi(r)) & \text{if $\lambda > 1$.}
    \end{cases}
\end{align}
We refer to \cite{BGT, Seneta} for further properties of regularly varying functions.

Next, we recall some basic facts about exact moduli of continuity \cite[\S7]{MR}.
Let $w, \tau: I \times I \to [0\,,\infty)$ be two continuous functions that vanish on the diagonal and are strictly positive near but off the diagonal.
We say that $w$ is an \emph{exact uniform modulus of continuity} for $X$ on $(I\,,\tau)$ if there is a constant $0<C<\infty$ such that
\[
    \lim_{h\to0^+} \sup_{t,s\in I: 0<\tau(t,s) \le h}\frac{|X(t)-X(s)|}{w(t\,,s)} = C \quad \text{a.s.}
\]
We say that $w$ is an \emph{exact local modulus of continuity} for $X$ at a fixed point $t_0 \in I$ if there is a constant $0<C<\infty$ such that
\[
    \lim_{h\to0^+} \sup_{t\in I: 0<\tau(t,t_0) \le h}\frac{|X(t)-X(t_0)|}{w(t\,,t_0)} = C \quad \text{a.s.}
\]
Let $m: [0\,,h_0] \to [0\,,\infty)$ be a continuous function such that $m(0)=0$ and $m(h) > 0$ for all $h \in (0\,,h_0]$.
We say that $m$ is an exact \emph{uniform m-modulus of continuity} for $X$ on $(I\,,\tau)$ if there is a constant $0<C<\infty$ such that
\[
    \limsup_{h\to0^+} \sup_{t,s\in I: 0<\tau(t,s) \le h} \frac{|X(t)-X(s)|}{m(h)} = C \quad \text{a.s.}
\]
We say that $m$ is an \emph{exact local m-modulus of continuity} for $X$ at a fixed point $t_0 \in I$ if there is a constant $0<C<\infty$ such that
\[
    \limsup_{h \to 0^+} \sup_{t \in I: 0<\tau(t,t_0)\le h} \frac{|X(t)-X(t_0)|}{m(h)} = C\quad \text{a.s.}
\]
In many cases, modulus and m-modulus are equivalent:

\begin{lemma}\cite[Lemma 7.1.6]{MR}\label{lem:m-mod}
    If $m$ is non-decreasing, then $m(\tau(t\,,s))$ is an exact uniform (resp.~local) modulus of continuity for $X$ if and only if $m(h)$ is an exact uniform (resp.~local) m-modulus of continuity for $X$.
    In this case, the exact modulus and m-modulus share the same constant $C$.
\end{lemma}

\begin{lemma}[Zero-one law]\label{lem:01}
    Let $w, \tau: I \times I \to [0\,,\infty)$ be two continuous functions that vanish on the diagonal and are strictly positive near but off the diagonal.
    Suppose $\sigma$ defined by \eqref{sigma} is continuous on $I \times I$.
    \begin{enumerate}[$(i)$]
        \item If 
        \begin{align}\label{01:um:sig/w}
            \lim_{h \to 0^+} \sup_{t,s \in I: 0<\tau(t,s)\le h} \frac{\sigma(t\,,s)}{w(t\,,s)} = 0,
        \end{align}
        then there is a non-random number $C \in [0\,,\infty]$ such that
        \begin{align}\label{01:um}
            \lim_{h \to 0^+} \sup_{t,s \in I: 0<\tau(t,s)\le h} \frac{|X(t)-X(s)|}{w(t\,,s)} = C \quad \text{a.s.}
        \end{align}
        \item Fix $t_0 \in I$. If
        \begin{align}\label{01:lm:sig/w}
            \lim_{h \to 0^+} \sup_{t \in I: 0<\tau(t,t_0)\le h} \frac{\sigma(t\,,t_0)}{w(t\,,t_0)} = 0,
        \end{align}
        then there is a non-random number $C(t_0) \in [0\,,\infty]$ such that
        \begin{align}\label{01:lm}
            \lim_{h \to 0^+} \sup_{t \in I: 0<\tau(t,t_0)\le h} \frac{|X(t)-X(t_0)|}{w(t\,,t_0)} = C(t_0) \quad \text{a.s.}
        \end{align}
    \end{enumerate}
\end{lemma}

\begin{proof}
This is an extension of \cite[Lemma 7.1.1]{MR} to more general forms of $w$ and to the case that $\tau$ is not necessarily a metric. We include a proof below for completeness.
First, continuity of $\sigma$ and Cauchy-Schwarz inequality imply that $X$ has a continuous covariance, so we may use Karhunen-Lo\`{e}ve theorem \cite[Theorem 5.3.2, Corollary 5.3.4]{MR} to write $X(t) = \sum_{n=1}^\infty f_n(t) \xi_n$ in $L^2(\Omega)$, where $f_n$ are continuous functions on $I$ and $\{\xi_n\}_{n=1}^\infty$ are i.i.d.~standard Gaussians.
Define $X_N(t) = \sum_{n=1}^N f_n(t) \xi_n$.
Since $\sigma^2(t\,,s) = \sum_{n=1}^\infty (f_n(t)-f_n(s))^2$, it follows that
\begin{align*}
    |X_N(t)-X_N(s)| \le \sum_{n=1}^N |f_n(t)-f_n(s)| \sup_{1 \le n \le N}|\xi_n| \le \sigma(t\,,s) \sup_{1 \le n \le N}|\xi_n|.
\end{align*}
From this, we see that if \eqref{01:um:sig/w} holds, then for every $N \ge 1$,
\begin{align*}
    \lim_{h\to0^+} \sup_{t,s\in I: 0<\tau(t,s) \le h} \frac{|X_N(t)-X_N(s)|}{w(t\,,s)} = 0 \quad \text{a.s.}
\end{align*}
This shows that the random variable on the left-hand side of \eqref{01:um} is measurable with respect to the tail sigma-algebra of $\{\xi_n\}_{n=1}^\infty$.
Hence, by Kolmogorov's 0-1 law, this random variable is constant a.s.
Similarly, we see that if \eqref{01:lm:sig/w} holds, then the random variable on the left-hand side of \eqref{01:lm} is constant a.s.
\end{proof}

Recall Slepian's lemma \cite{S62} (see also \cite[Lemma 5.5.1]{MR}):

\begin{lemma}\label{lem:slepian}
Let $Y$ and $Z$ be centered $\R^m$-valued Gaussian random variables.
If $\E[Y_i^2]=\E[Z_i^2]$ and $\E[Y_iY_j] \le \E[Z_i Z_j]$ for all $i,j\in \{1,\dots, m\}$, then
\begin{align*}
    \textstyle
    \P\left( \bigcup_{j=1}^m \{ Y_j > a_j \} \right) \le \P\left( \bigcup_{j=1}^m \{Z_j > a_j\} \right)
\end{align*}
for any non-random numbers $a_1,\dots, a_n \in \R$.
\end{lemma}

\begin{lemma}\label{lem:cv:cov}
Let $Y$ and $Z$ be centered jointly Gaussian random variables. 
Let $C_0 \in [0\,,1]$ and consider the following statements:
\begin{enumerate}[$(a)$]
    \item $\Var(Y|Z) \ge C_0 \Var(Y)$;\smallskip
    \item $|\E[YZ]| \le \sqrt{1-C_0} \, \|Y\|_2 \|Z\|_2$;\smallskip
    \item $\E[YZ] \le \sqrt{1-C_0} \, \|Y\|_2 \|Z\|_2$.\smallskip
\end{enumerate}
We have $(a) \Leftrightarrow (b) \Rightarrow (c)$.
\end{lemma}

\begin{proof}
$(b)\Rightarrow (c)$ is obvious. The equivalence of $(a)$ and $(b)$ follows from the identity
\[
    \Var(Y|Z) = \Var(Y) - \frac{|\mathrm{Cov}(Y,Z)|^2}{\Var(Z)}
\]
for any jointly Gaussian random variables $Y$ and $Z$.
\end{proof}

\section{Proofs of general results}
\label{s:pf:general}

\subsection{Proof of Theorem \ref{th:um:aniso}}

Set 
\[\textstyle
    w(t\,,s) = \phi(t-s) \sqrt{\log_+ \left[1/\psi(\phi(t-s))\right]}.
\]
Then \eqref{01:um:sig/w} is satisfied under Assumption \ref{a:sigma:aniso}, so we may apply Lemma \ref{lem:01}(i) to deduce that \eqref{um:aniso} holds for some constant $C \in [0\,,\infty]$.

We first prove the upper bound for $C$ under Assumptions \ref{a:sigma:aniso} and \ref{a:psi:aniso}.
Let $\mathcal{N}(\varepsilon) = \mathcal{N}(I\,,\sigma\,,\varepsilon)$ be the smallest number of $\sigma$-balls with radius $\varepsilon$ needed to cover $I$.
By Assumption \ref{a:sigma:aniso}, there is $c_0>1$ such that for all $\varepsilon>0$ small,
\[
    \mathcal{N}(\varepsilon) \le \frac{c_0|I|}{\prod_{i=1}^N \phi_i^{-1}(\varepsilon/C_2)} =  \frac{c_0|I|}{\psi(\varepsilon/C_2)}
\]
where $|I|$ denotes the volume of $I$.
By \cite[Theorem 1.3.5]{AT}, we can find a universal constant $0<K<\infty$ and a random variable $\eta \in (0\,,\infty)$ such that almost surely,
\begin{align*}
    \sup_{t,s \in I: \sigma(t,s) \le h}|X(t)-X(s)| \le K \int_0^h \sqrt{\log \mathcal{N}(\varepsilon)} \, \d \varepsilon \qquad \forall h \in (0\,,\eta].
\end{align*}
Then, we may use \eqref{sigma:phi}, the change of variable $r = \psi(\varepsilon/C_2)$, integration by parts, followed by \eqref{psi:1} and \eqref{psi:2}, to deduce that
\begin{align*}
    &\sup_{t,s\in I : \phi(t,s) \le h} |X(t)-X(s)| 
    \le KC_2 \int_0^{\psi(h)} \sqrt{\log\left(\frac{c_0|I|}{r}\right)} \, \d \psi^{-1}(r)\\
    & = K C_2 h \sqrt{\log \left(\frac{c_0|I|}{\psi(h)}\right)} + \tfrac12 KC_2 \int_0^{\psi(h)} \frac{\psi^{-1}(r)}{r\sqrt{\log(c_0|I|/r)}} \, \d r\\
    & \le KC_2\left(1 + \tfrac12 C_3 + o(1)\right) h \sqrt{\log\left( \frac{1}{\psi(h)}\right)} \qquad \text{as $h \to 0^+$}.
\end{align*}
It follows that
\[
    \lim_{h \to 0^+} \sup_{t,s\in I: 0<\phi(t,s) \le h}\frac{|X(t)-X(s)|}{h\sqrt{\log(1/\psi(h))}} \le KC_2(1+C_3/2) \quad \text{a.s.}
\]
By \cite[Lemma 7.1.6]{MR}, the preceding implies that
\[
    \lim_{h \to 0^+} \sup_{t,s\in I: 0<\phi(t,s) \le h}\frac{|X(t)-X(s)|}{\phi(t-s)\sqrt{\log(1/\psi(\phi(t-s)))}} \le KC_2(1+C_3/2) \quad \text{a.s.}
\]
This proves that $C\le KC_2(1+C_3/2) < \infty$.

Next, we show the lower bound for $C$ provided in addition that Assumption \ref{a:cv:aniso} holds.
In particular, we can choose a sequence $h_n \downarrow 0$ and for each $n \in \N$, a set of points $( t_{n,i}\,, s_{n,i})_{i=1}^m$ in $I$ such that:
\begin{enumerate}[(1)]
    \item $\phi(t_{n,i}-s_{n,i}) = h_n$ for $i=1,\dots,m$;
    \item $\phi(s_{n,i}-s_{n,j}) \ge \psi^{-1}(A_0 (\psi(h_n))^\rho)$ for any distinct $i,j \in \{1,\dots, m\}$;
    \item $m = m(n) \ge B_0 (\psi(h_n))^{-\rho}$; and
    \item for every pair of distinct $i,j \in \{1,\dots, m\}$,
    \begin{align}\begin{split}\label{cv:ij}
        &\E[(X(t_{n,i})-X(s_{n,i}))(X(t_{n,j})-X(s_{n,j}))]\\
        &\qquad \le \sqrt{1-C_0} \|X(t_{n,i})-X(s_{n,i})\|_2 \|X(t_{n,j})-X(s_{n,j})\|_2.
    \end{split}\end{align}
\end{enumerate}
In fact, by choosing a maximal set of points with $\phi(s_{n,i}-s_{n,j}) \ge \psi^{-1}(A_0(\psi(h_n))^\rho)$, we see that the $\phi$-balls centered at $s_{n,i}$ with radius $\psi^{-1}(A_0(\psi(h_n))^\rho)$ form a covering of $I$. Since each ball has volume $A_0(\psi(h_n))^\rho$, this shows that $m$ may be taken as large as the number of such balls, which is $\asymp (\psi(h_n))^{-\rho}$.

Let $\xi_0, \xi_1,\dots$ be i.i.d.~standard Gaussian random variables, and define
\begin{align*}
    Z_i = (1-C_0)^{1/4} \xi_0 + ( 1- \sqrt{1-C_0})^{1/2} \xi_i \quad \text{for $i \in \N_0$.}
\end{align*}
For each $n \in \N$, define
\begin{align*}
    Y_{n,i} = \frac{X(t_{n,i})-X(s_{n,i})}{\sigma(t_{n,i}\,,s_{n,i})} \quad \text{for $i=1,\dots, m$.}
\end{align*}
Then, we have
\begin{align*}
    \E[Y_{n,i}^2] = 1 = \E[Z_{i}^2] \quad \text{for $i=1,\dots, m$.}
\end{align*}
Also, it follows from \eqref{cv:ij} that 
\begin{align*}
    \E[Y_{n,i}Y_{n,j}] \le \sqrt{1-C_0} = \E[Z_{i}Z_{j}] \quad \text{for $i\ne j$.}
\end{align*}
Therefore, an application of Slepian's lemma (Lemma \ref{lem:slepian}) for any large but fixed $n\in \N$ and any fixed number $c>0$ yields
\begin{align}\label{slepian:Y:Z}
    \P\left( \bigcup_{i=1}^m \left\{ Y_{n,i} > c \sqrt{\log(1/\psi(h_n))}\right\} \right) 
    \ge \P\left( \bigcup_{i=1}^m \left\{Z_i > c \sqrt{\log(1/\psi(h_n))} \right\} \right).
\end{align}
Choose and fix a number $a$ such that
\begin{align}\label{um:aniso:a}
    0<a< C_1\left[ 2\rho \, \big(1-\sqrt{1-C_0}\big)\right]^{1/2}.
\end{align}
We claim that if $\delta>0$ is sufficiently small, then
\begin{align}\label{slepian:claim}
    \lim_{n\to\infty} \P\left( \bigcup_{i=1}^m \left\{ \xi_i > \frac{(1+\delta) aC_2\sqrt{\log(1/\psi(h_n))}}{(1-\sqrt{1-C_0})^{1/2}}  \right\} \right) = 1.
\end{align}
In fact, due to \eqref{um:aniso:a}, we may choose and fix $\delta>0$ sufficiently small such that 
\begin{align}\label{um:aniso:delta}
    (1+\delta)a < C_1 \left[ 2\rho\, \big(1-\sqrt{1-C_0}\big)\right]^{1/2}.
\end{align}
Then, 
by the independence of $(\xi_i)_{i=1}^m$, the standard Gaussian tail estimate 
\begin{align}\label{gaus:tail}
    \frac{z^{-1} e^{-z^2/2}}{2\sqrt{2\pi}}  \le \P\{\xi_1 > z\} \le \frac{e^{-z^2/2}}{\sqrt{2\pi}}  \quad \forall z \ge 1,
\end{align}
the inequality $1-x \le e^{-x}$ for all $x \ge 0$, and the bound $m \ge B_0 (\psi(h_n))^{-\rho}$, we have for sufficiently large $n$ that
\begin{align*}
    &\P\left( \bigcup_{i=1}^m \left\{ \xi_i > \frac{(1+\delta)a\sqrt{\log(1/\psi(h_n))}}{C_1(1-\sqrt{1-C_0})^{1/2}}  \right\} \right)\\
    & = 1 - \prod_{i=1}^m\left[ 1 - \P\left\{ \xi_1 > \frac{(1+\delta)a\sqrt{\log(1/\psi(h_n))}}{C_1(1-\sqrt{1-C_0})^{1/2}} \right\} \right]\\
    & \ge 1 - \exp\left( - \frac{m}{\sqrt{2\pi}} \frac{C_1(1-\sqrt{1-C_0})^{1/2}}{2 (1+\delta)a\sqrt{\log(1/\psi(h_n))}} \exp\left(-\frac{(1+\delta)^2 a^2 \log(1/\psi(h_n))}{2C_1^2(1-\sqrt{1-C_0})}\right) \right)\\
    & \ge 1 - \exp\left( - \frac{B_0}{\sqrt{2\pi}} \frac{C_1(1-\sqrt{1-C_0})^{1/2}}{2 (1+\delta)a\sqrt{\log(1/\psi(h_n))}}
    (\psi(h_n))^{\frac{(1+\delta)^2 a^2 }{2C_1^2(1-\sqrt{1-C_0})}-\rho} \right).
\end{align*}
Thanks to \eqref{um:aniso:delta}, the exponential term tends to 0 as $n \to \infty$. This yields \eqref{slepian:claim}.

We may now complete the proof of the theorem as follows. By \eqref{sigma:phi} and \eqref{slepian:Y:Z}, it follows that for sufficiently large $n$,
\begin{align*}
    &\P(E_n):=\P\left\{ \sup_{t, s\in I: 0<\phi(t-s) \le h_n} \frac{|X(t)-X(s)|}{\phi(t-s)\sqrt{\log(1/\psi(\phi(t-s)))}} \ge a \right\}\\
    & \ge \P\left( \bigcup_{i=1}^m \left\{ Y_{n,i} > \frac{a}{C_1} \sqrt{\log(1/\psi(h_n))} \right\}\right)
    \ge \P\left( \bigcup_{i=1}^m \left\{ Z_{i} > \frac{a}{C_1} \sqrt{\log(1/\psi(h_n))} \right\}\right)\\
    & = \P\left\{ (1-C_0)^{1/4} \xi_0 + (1-\sqrt{1-C_0})^{1/2} \max_{1\le i \le m} \xi_i > \frac{a}{C_1} \sqrt{\log(1/\psi(h_n))} \right\}\\
    & \ge \P\left\{ \max_{1\le i \le m} \xi_i > \frac{(1+\delta)a  \sqrt{\log(1/\psi(h_n))}}{C_1(1-\sqrt{1-C_0})^{1/2}} \right\} - \P\left\{ \xi_0 < -\frac{\delta a  \sqrt{\log(1/\psi(h_n))}}{C_1(1-C_0)^{1/4}} \right\}.
\end{align*}
Then, we may let $n \to \infty$, use monotonicity of the events $E_n$, and apply \eqref{slepian:claim} and $\sqrt{\log(1/\psi(h_n))} \to +\infty$ to the last two probabilities respectively to deduce that
\begin{align*}
    &\P\left\{ \lim_{n\to \infty}\sup_{t, s\in I: 0<\phi(t-s) \le h_n} \frac{|X(t)-X(s)|}{\phi(t-s)\sqrt{\log(1/\psi(\phi(t-s)))}} \ge a\right\}\\
    &= \lim_{n\to\infty}\P(E_n)
    \ge \lim_{n \to \infty}\P\left\{ \max_{1\le i \le m} \xi_i > \frac{(1+\delta) a \sqrt{\log(1/\psi(h_n))}}{C_1(1-\sqrt{1-C_0})^{1/2}} \right\} - 0 = 1.
\end{align*}
This proves that $C \ge a$ for every $a$ satisfying \eqref{um:aniso:a}.
Therefore, this implies that $C \ge C_1 [2\rho\,(1-\sqrt{1-C_0})]^{1/2}$ and completes the proof of the theorem.
\qed

\subsection{Proof of Theorem \ref{th:lm:aniso}}

Fix $t_0\in I$ and set 
\[
    w(t\,,t_0) = \phi(t-t_0)\sqrt{\ell(1/\psi(\phi(t-t_0)))}.
\]
Since $\lim_{x\to\infty} \ell(x) = \infty$, Assumption \ref{a:sigma:aniso} and Lemma \ref{lem:01}(ii) imply that \eqref{lm:aniso} holds for some constant $C=C(t_0) \in [0\,,\infty]$.

We first prove the upper bound for $C$ under Assumptions \ref{a:sigma:aniso} and \ref{a:psi:lil}.
Fix a sequence $q_n \downarrow 0$, and 
define $S_n = \{ t \in I : \phi(t-t_0) \le q_n \}$ for each $n \in \N$.
Let $Y(t) = X(t)-X(t_0)$. Then $\|Y(t)-Y(s)\|_2 = \sigma(t\,,s)$ for all $t,s \in I$.
Let $\mathcal{N}_n(\varepsilon) = \mathcal{N}(S_n\,,\sigma\,,\varepsilon)$ be the smallest number of $\sigma$-balls with radius $\varepsilon$ needed to cover $S_n$.
Under Assumption \ref{a:sigma:aniso}, the $\sigma$-diameter of $S_n$ satisfies $D_n = \sup_{t \in S_n}\sigma(t\,,t_0) \le C_2 q_n$ and for $n$ large,
\[
    \mathcal{N}_n(\varepsilon)
    \le \frac{2 \psi(q_n)}{\psi(\varepsilon/C_2)} \quad \forall \varepsilon \in (0\,,D_n].
\]
By Dudley's theorem \cite{Dudley} (see also \cite[Theorem 1.3.3]{AT} or \cite[Theorem 6.1.2]{MR}), there is a universal constant $0<K<\infty$ such that
\begin{align*}
    \E\sup_{t\in S_n}|X(t)-X(t_0)|=\E\sup_{t\in S_n}|Y(t)|
    &\le K\int_0^{D_n} \sqrt{\log\mathcal{N}_n(\varepsilon)} \, \d \varepsilon\\
    &\le K \int_0^{C_2 q_n} \sqrt{\log\left( \frac{2\psi(q_n)}{\psi(\varepsilon/C_2)} \right)} \, \d \varepsilon.
\end{align*}
Then, we may use the change of variable $r = \psi(\varepsilon/C_2)$, integration by parts, and Assumption \ref{a:psi:lil} to deduce that
\begin{align*}
    &\E\sup_{t\in S_n}|X(t)-X(t_0)|
    \le KC_2 \int_0^{\psi(q_n)} \sqrt{\log\left( \frac{2\psi(q_n)}{r} \right)} \, \d \psi^{-1}(r)\\
    &= KC_2 q_n \sqrt{\log 2} + KC_2 \int_0^{\psi(q_n)} \frac{\psi^{-1}(r)}{2r\sqrt{\log(2 \psi(q_n)/r)}} \, \d r\\
    & = KC_2 q_n \sqrt{\log 2} + KC_2  \int_0^1 \frac{\psi^{-1}(\psi(q_n) r)}{2r\sqrt{\log(2/r)}} \, \d r\\
    &\le (KC_2C_4/2 +o(1))q_n\sqrt{\ell (1/ \psi(q_n)) } \quad \text{as $n\to\infty$.}
\end{align*}
By Borell's inequality \cite{Borell} (see also \cite[Theorem 2.1.1]{AT}) and the preceding, it follows that for any fixed $a>0$, $\delta>0$ and sufficiently large $n$,
\begin{align*}
    p_n&:=\P\left\{ \sup_{t \in I: q_{n+1}\le \phi(t-t_0)\le q_n} \frac{|X(t)-X(t_0)|}{\phi(t-t_0)\sqrt{\ell(1/ \psi(\phi(t-t_0)))}} > (a+KC_2C_4/2)(1+\delta) \right\}\\
    & \le \P\left\{ \sup_{t \in I: \phi(t-t_0) \le q_n} \frac{|X(t)-X(t_0)|}{q_{n+1}\sqrt{\ell(1/ \psi(q_n))}} > (a+KC_2C_4/2)(1+\delta) \right\}\\
    &\le \P\left\{ \sup_{t \in S_n}|X(t)-X(t_0)| - \E\sup_{t \in S_n}|X(t)-X(t_0)| > a q_{n+1} \sqrt{\ell(1/ \psi(q_n)) } \right\}\\
    & \le 2 \exp\left( - \frac{a^2 q_{n+1}^2\ell(1/ \psi(q_n)) }{2 \sup_{t \in S_n}\sigma^2(t\,,t_0)} \right)
    \le 2 \exp\left(-\frac{a^2 q_{n+1}^2 \ell(1/ \psi(q_n)) }{2C_2^2 q_n^2}\right),
\end{align*}
where we have used \eqref{sigma:phi} to obtain the last inequality.
Since each $\phi_i$ is regularly varying at 0 of index $\theta_i \in [0\,,1]$, for fixed $\epsilon>0$, we have $\phi_i(r) \gtrsim r^{\theta_i+\epsilon}$ as $r \to 0^+$ for every $i$ (see \cite[Prop 1.3.6(v)]{BGT} or \cite[p.18]{Seneta}), and hence monotonicity of $\phi_i$ implies that
$\phi_i^{-1}(r) \lesssim r^{1/(\theta_i+\epsilon)}$ and $\psi(r) \lesssim r^{\nu_\epsilon}$ as $r \to 0^+$, where $\nu_\epsilon = \sum_{i=1}^N 1/(\theta_i+\epsilon)$.
We now consider the following two cases:

Case (1): $\ell(x) = \log\log(x)$ and $h_n = e^{-\rho n^\gamma}$, where $\rho>0$ and $\gamma>0$.
In this case, fix a number $q>1$ and take $q_n = q^{-n}$.
Then, for a sufficiently large $n_0$,
\begin{align*}
    \sum_{n=n_0}^\infty p_n 
    \lesssim 2\sum_{n=n_0}^\infty \left[ \log(1/\psi(q^{-n})) \right]^{-\frac{a^2}{2C_2^2q^2}}
    \lesssim \sum_{n=n_0}^\infty  n^{-\frac{a^2}{2C_2^2 q^2}},
\end{align*}
which is finite provided that $a>\sqrt{2} C_2 q$.

Case (2): $\ell(x) = \log(x)$ and $h_n = n^{-\gamma}$, where $\gamma>0$.
In this case, take $q_n = 1/n$. 
Then, for a sufficiently large $n_0$,
\begin{align*}
    \sum_{n=n_0}^\infty p_n \le 2 \sum_{n=n_0}^\infty \left[\psi(1/n)\right]^{\frac{a^2n^2}{2C_2^2(n+1)^2}} \lesssim \sum_{n=n_0}^\infty n^{-\frac{\nu_\epsilon a^2n^2}{2C_2^2(n+1)^2}},
\end{align*}
which is summable provided that $a>\sqrt{2}C_2/\sqrt{\nu_\epsilon}$.

In both cases, an application of the Borel-Cantelli lemma and the arbitrariness of $\delta>0$ imply that
\[
    \lim_{n\to\infty} \sup_{t \in I: 0<\phi(t-t_0) \le q_n} \frac{|X(t)-X(t_0)|}{\phi(t-t_0)\sqrt{\ell(1/\psi(\phi(t-t_0)))}} \le a+KC_2C_4/2 \quad \text{a.s.}
\]
and hence $C\le a+KC_2C_4/2$.
This shows \eqref{lm:C:H<1}.

Next, we prove the lower bound for $C$ under the additional Assumption \ref{a:cv:lil}.
Let $(t_n)_{n=1}^\infty$ be the sequence in $I$ given by Assumption \ref{a:cv:lil}.
In particular, there exist $c_1 \in (0\,,1)$ and $c_2>1$ such that $c_1 h_n \le \psi(\phi(t_n-t_0)) \le c_2 h_n$ for all $n \ge n_0$.
Let $\xi_0,\xi_1,\dots$ be i.i.d.~standard Gaussian random variables.
For each $n \in \N$, define
\begin{align}\label{lil:Y:Z}
    Y_n = \frac{X(t_n)-X(t_0)}{\sigma(t_n\,,t_0)}, \quad
    Z_n = (1-C_0)^{1/4}\xi_0 + (1-\sqrt{1-C_0})^{1/2} \xi_n.
\end{align}
Then $\E[Y_n^2] = 1 = \E[Z_n^2]$ and
\[
    \E[Y_nY_m] \le \sqrt{1-C_0} = \E[Z_nZ_m] \quad \text{for all $m>n\ge n_0$,}
\]
thanks to Assumption \ref{a:cv:lil} and Lemma \ref{lem:cv:cov}.
Hence, we may apply Slepian's lemma (Lemma \ref{lem:slepian}) to deduce that for any fixed number $c>0$ and $M>m\gg n_0$,
\begin{align}\label{slepian:lil}
    \P\left( \bigcup_{n=m}^M \left\{ Y_n >  c \sqrt{\ell(\tfrac{1}{c_1 h_n})}\right\} \right) \ge \P \left( \bigcup_{n=m}^M \left\{ Z_n >  c \sqrt{\ell (\tfrac{1}{c_1 h_n})}\right\}  \right).
\end{align}
By independence of $(\xi_n)_{n=1}^\infty$, the Gaussian tail estimate \eqref{gaus:tail} and the inequality $1-x \le e^{-x}$ for all $x \ge 0$, we have for $M>m \gg n_0$ that
\begin{align}\begin{split}\label{P:xi:lil}
    \P\left(\bigcup_{n=m}^M \left\{ \xi_n > c \sqrt{\ell(1/h_n)}\right\} \right)
    &= 1 - \prod_{n=m}^M \left[ 1 - \P\left\{ \xi_n > c \sqrt{\ell(1/h_n)} \right\} \right]\\
    &\ge 1 - \exp\left( - \frac{1}{2\sqrt{2\pi}} \sum_{n=m}^M \frac{\exp(-\frac{c^2}{2}\ell(1/h_n))}{c\sqrt{\ell(1/h_n)}}  \right).
\end{split}\end{align}
Now, it follows from \eqref{sigma:phi}, the fact that $\psi(\phi(t_n-t_0)) \ge c_1 h_n$, the fact that $\ell$ is increasing, and \eqref{slepian:lil} that for any fixed $a>0$, $\delta>0$, and $M>m \gg n_0$,
\begin{align*}
    &\P(F_{m,M}):=\P\left\{ \sup_{t \in I: c_1 h_M\le \psi(\phi(t-t_0))\le c_2 h_m} \frac{|X(t)-X(t_0)|}{\phi(t-t_0)\sqrt{\ell(1/\psi(\phi(t-t_0)) )}} \ge a \right\}\\
    & \ge \P\left\{ \max_{m \le n \le M} \frac{X(t_n)-X(t_0)}{\sigma(t_n\,,t_0) \sqrt{\ell(\tfrac{1}{c_1 h_n})}} > \frac{a}{C_1} \right\}\\
    &\ge \P\left( \bigcup_{n=m}^M \left\{ Y_n > \frac{a}{C_1}\sqrt{\ell(\tfrac{1}{c_1 h_n})} \right\} \right) 
    \ge \P\left( \bigcup_{n=m}^M \left\{ Z_n > \frac{a}{C_1}\sqrt{\ell(\tfrac{1}{c_1 h_n})} \right\} \right)\\
    & \ge \P\left( \bigcup_{n=m}^M \left\{\xi_n > \frac{(1+\delta)a\sqrt{\ell(1/h_n)}}{C_1(1-\sqrt{1-C_0})^{1/2}}\right\} \right) - \P\left\{ \xi_0 < - \frac{\delta a \sqrt{\ell(1/h_m)}}{C_1(1-C_0)^{1/4}}\right\}.
\end{align*}
In both Case (1) and Case (2), plugging in the estimate in \eqref{P:xi:lil} implies that there is $c>0$ such that for sufficiently large $M>m\gg n_0$,
\[
    \P(F_{m,M}) \ge 1 - \exp\left( -c  \sum_{n=m}^M \frac{n^{-\frac{\gamma (1+\delta)^2a^2}{2C_1^2(1-\sqrt{1-C_0})}}}{\sqrt{\log ( n^\gamma)}} \right) - \P\left\{ \xi_0 < - c \sqrt{\log m}\right\}.
\]
Hence, for any fixed $\delta>0$ and $a>0$ with
\begin{align}\label{a:lil}
    a < \frac{C_1}{1+\delta} \left[ 2\gamma^{-1} \big(1-\sqrt{1-C_0}\big) \right]^{1/2},
\end{align}
we may let $M\to\infty$ first, then $m\to\infty$ to deduce that
\begin{align*}
    \lim_{m\to\infty} \sup_{t \in I: 0<\phi(t-t_0) \le \psi^{-1}(c_2 h_m)} \frac{|X(t)-X(t_0)|}{\phi(t-t_0)\sqrt{\ell(1/\psi(\phi(t-t_0)))}} \ge a \quad \text{a.s.}
\end{align*}
Since the last display holds for arbitrary $\delta>0$ and $a$ satisfying \eqref{a:lil}, we conclude that $C \ge C_1 [2\gamma^{-1}(1-\sqrt{1-C_0})]^{1/2}$. 
This completes the proof of Theorem \ref{th:lm:aniso}.
\qed

\section{Proofs of Theorems \ref{th:um} and \ref{th:lm}}
\label{s:pf:main}

\subsection{Proof of Theorem \ref{th:um}}

First, Lemma \ref{lem:01}(i) implies that \eqref{um} holds for some constant $K_1 \in [0\,,\infty]$.
We aim to show that $0<K_1<\infty$.

Let $\phi_i(r)=\varphi(r)$ and $\phi(t-s) = \max_{1\le i \le N}\varphi(|t^i-s^i|)$.
From \eqref{psi}, we have
\begin{align}\label{psi:psi-1}
    \psi(r) = (\varphi^{-1}(r))^N \quad \text{and} \quad \psi^{-1}(r) = \varphi(r^{1/N}).
\end{align}
By \eqref{rv:bd}, there exists $c_0>1$ such that $\varphi(\sqrt{N}r) \le c_0\varphi(r)$ for all $r>0$ small.
Then, monotonicity of $\varphi$ and the preceding imply that
\begin{align}\begin{split}\label{phi:phi:phi}
    \phi(t-s) \le \varphi(|t-s|) &
    \le \max_{1\le i \le N}\varphi\big(\sqrt{N} |t^i-s^i|\big)
    \le c_0\phi(t-s)
\end{split}\end{align}
for all $t,s\in I$ with $|t-s|$ small.
Hence, Assumptions \ref{a:sigma:aniso} and \ref{a:psi:aniso} are satisfied.
Next, we verify Assumption \ref{a:cv:aniso} in order to apply Theorem \ref{th:um:aniso}.
By the current assumption of Theorem \ref{th:um}, there exists $\varepsilon_0 \in (0\,,1]$ such that \eqref{um:cv} holds for all $t,s,t',s' \in I$ with 
\begin{align}\label{t-s:t'-s'}
    |t-s| \le \varepsilon, \qquad |t'-s'| \le \varepsilon
\end{align}
and
\begin{align}\label{s-s'}
    |s-s'| \ge a_0 \varepsilon^\rho
\end{align}
as long as $\varepsilon \in (0\,,\varepsilon_0]$.
Set $h=c_0^{-1}\varphi(\varepsilon)$.
By \eqref{phi:phi:phi}, for $h>0$ small, $\phi(t-s) \le h$ and $\phi(t'-s') \le h$ imply \eqref{t-s:t'-s'}.
By \eqref{rv:bd:inv}, we can find $c_1, c_2, r_1>0$ such that
\begin{align}\label{phi:inv:bd}
    \varphi^{-1}(c_2 \varphi(r)) \le c_1 r \le \varphi^{-1}(c_0^{-1}\varphi(r)) \quad \forall r \in (0\,,r_1].
\end{align}
If we choose $A_0=a_0^N c_1^{-N\rho}$, then it follows from \eqref{psi:psi-1}, \eqref{phi:phi:phi}, \eqref{phi:inv:bd} that $\phi(s-s') \ge \psi^{-1}(A_0(\psi(h))^\rho)$ implies \eqref{s-s'} for $h>0$ small:
\begin{align*}
    |s-s'| \ge A_0^{1/N} (\varphi^{-1}(h))^{\rho} = \left(A_0^{1/(N\rho)} \varphi^{-1}(c_0^{-1}\varphi(\varepsilon))\right)^\rho \ge a_0\varepsilon^\rho.
\end{align*}
Hence, \eqref{um:cv} holds for all $t,s,t',s'\in I$ with $\phi(t-s) \le h$, $\phi(t'-s') \le h$ and $\phi(s-s') \ge \psi^{-1}(A_0(\psi(h))^\rho)$ provided that $h>0$ is sufficiently small.
Thanks to Lemma \ref{lem:cv:cov}, we have thus verified Assumption \ref{a:cv:aniso}.
Therefore, by Theorem \ref{th:um:aniso} and the first identity in \eqref{psi:psi-1},
\[
    \lim_{h\to0^+} \sup_{t,s\in I: 0<\phi(t-s) \le h} \frac{|X(t)-X(s)|}{\phi(t-s)\sqrt{\log[1/\varphi^{-1}(\phi(t-s))]}} = C\sqrt{N} \quad \text{a.s.}
\]
for some constant $0<C<\infty$.
From this, we can use monotonicity of $\varphi^{-1}$ together with \eqref{phi:phi:phi} and \eqref{phi:inv:bd} to deduce that $\varphi^{-1}(\phi(t-s)) \asymp |t-s|$ as $|t-s|\to0^+$ and
\[
    0<\lim_{h\to0^+} \sup_{t,s\in I: 0<|t-s| \le h} \frac{|X(t)-X(s)|}{\varphi(|t-s|)\sqrt{\log(1/|t-s|)}} <\infty \quad \text{a.s.}
\]
This shows that $0<K_1<\infty$ and completes the proof of Theorem \ref{th:um}.
\qed

\subsection{Proof of Theorem \ref{th:lm}}

First, Lemma \ref{lem:01}(ii) implies that \eqref{lm} holds for some constant $K_2 \in [0\,,\infty]$. It remains to show that $0<K_2<\infty$.

Again, let $\phi_i(r)=\varphi(r)$ and $\phi(t-s) = \max_{1\le i \le N}\varphi(|t^i-s^i|)$.
By \eqref{psi:psi-1}, Assumptions \ref{a:sigma:aniso} and \ref{a:psi:lil} are satisfied with $\ell(x) = \log\log(x)$.
We verify Assumption \ref{a:cv:lil} next.
By the underlying assumption of Theorem \ref{th:lm}, there is a sequence $(t_n)_{n=n_0}^\infty$ in $I$ with $|t_n-t_0|=a^{-n}$ such that \eqref{lm:cv} holds.
In particular, by \eqref{psi:psi-1} and monotonicity of $\varphi$, we have
\begin{align*}
    \psi(\phi(t_n-t_0)) \le (\varphi^{-1}(\varphi(|t_n-t_0|)))^N = |t_n-t_0|^N = a^{-Nn}
\end{align*}
and
\begin{align*}
    \psi(\phi(t_n-t_0)) 
    &= \Big[\varphi^{-1}\Big(\max_{1\le i \le N}\varphi(|t_n^i-t_0^i|)\Big)\Big]^N 
    = \max_{1\le i \le N} \Big[\varphi^{-1}(\varphi(|t_n^i-t_0^i|))\Big]^N\\
    & \ge \max_{1\le i \le N} |t_n^i-t_0^i|^N \ge (N^{-1/2} |t_n-t_0| )^N = N^{-N/2} a^{-Nn}
\end{align*}
for all $n \ge n_0$.
This verifies Assumption \ref{a:cv:lil} with $h_n = e^{-(N \log a) n}$, thanks to Lemma \ref{lem:cv:cov}.
Therefore, we may apply Case (1) of Theorem \ref{th:lm:aniso} to deduce that
\[
    \lim_{h \to0^+} \sup_{t \in I: 0<\phi(t-t_0) \le h} \frac{|X(t)-X(t_0)|}{\phi(t-t_0)\sqrt{\log\log[1/\varphi^{-1}(\phi(t-t_0))]}} = C \quad \text{a.s.}
\]
for some constant $0<C<\infty$.
This, together with monotonicity of $\varphi^{-1}$, \eqref{phi:phi:phi} and \eqref{phi:inv:bd}, implies that
\[
    0<\lim_{h\to0^+}\sup_{t \in I: 0<|t-t_0| \le h} \frac{|X(t)-X(t_0)|}{\varphi(|t-t_0|)\sqrt{\log\log(1/|t-t_0|)}} < \infty \quad \text{a.s.}
\]
Hence, $0<K_2 < \infty$ and the proof of Theorem \ref{th:lm} is complete.

\section{Linear SPDEs}
\label{s:SPDE}

In this section, we discuss an application of our results to a class of linear SPDEs.
Choose and fix parameters
\begin{align*}
    \alpha>0, \quad 0<\nu\le 1, \quad 0< \ell \le d,
\end{align*}
and consider the SPDE
\begin{align}\label{SPDE:general}
    \begin{cases}
        \partial_t u(t\,,x) = -(-\Delta)^{\alpha/2} u(t\,,x) + \dot{W}(t\,,x), \quad t \ge 0, x \in \R^d,\\
        u(0\,,\cdot) = 0,
    \end{cases}
\end{align}
where $-(-\Delta)^{\alpha/2}$ denotes fractional Laplacian and $\dot{W}$ is a centered generalized Gaussian noise with covariance
\begin{align}\label{W:cov}
    \E[\dot{W}(t\,,x)\dot{W}(s\,,y)] = \Lambda_{1,\nu}(t-s) \Lambda_{d,\ell}(x-y),
\end{align}
where
\begin{align*}
    \Lambda_{n,\ell}(z) 
    := \begin{cases}
    \delta(z) & \text{if $\ell=n$,}\\
    |z|^{-\ell} & \text{if $0<\ell<n$,}
    \end{cases}
    \qquad \text{for $z \in \R^n$.}
\end{align*}
Note that
\begin{align*}
    \nu = 2-2H, \qquad \text{where $1/2 \le H<1$,}
\end{align*}
corresponds to a Gaussian noise that is fractional in time with Hurst index $H$.
According to standard SPDE theory \cite{Walsh, Dalang}, \eqref{SPDE:general} has a pointwise defined mild solution 
\begin{align}\label{mild:sol}
    u(t\,,x) = \int_0^t \int_{\R^d} G(t-s\,,x-y) W(\d s\, \d y) \qquad \forall t > 0, x \in \R^d,
\end{align}
if and only if the above Wiener integral is well defined pointwise.
In \eqref{mild:sol}, $G$ denotes the fundamental solution of $\partial_t + (-\Delta)^{\alpha/2}$ whose Fourier transform in the spatial variable is
\begin{align}\label{G:hat}
    \hat{G}(t,\xi) = e^{-t |\xi|^\alpha}
    \qquad \forall t > 0, \xi \in \R^d.
\end{align}
By \cite{HSWX20, CLX}, a pointwise defined solution \eqref{mild:sol} exists if and only if
\begin{align}\label{ns:cond:sol}
    \alpha > \frac{\ell}{2-\nu} = \frac{\ell}{2H}.
\end{align}
In this case, the solution $u$ is a pointwise defined centered Gaussian random field.

Throughout the remainder of this section, we assume condition \eqref{ns:cond:sol} and discuss the spatial and temporal regularities of the solution.
Set
\begin{align*}
    \theta_1 := H-\frac{\ell}{2\alpha}
    \qquad \text{and} \qquad 
    \theta_2 
    := \alpha H - \frac{\ell}{2}
    = \alpha \theta_1.
\end{align*}
Note that $0<\theta_1<1$ and $0<\theta_2<\alpha$.
According to \cite{HSWX20, LX23, CLX}, the following regularity result holds.

\begin{lemma}\label{lem:SPDE:reg}
    Fix a compact rectangle $I \times J \subset (0\,,\infty) \times \R^d$.
    Then, under \eqref{ns:cond:sol}, we have 
    \begin{align*}
        \|u(t\,,x)-u(s\,,y)\|_2 \asymp \phi((t\,,x)\,,(s\,,y)) := \phi_1(|t-s|) \vee \phi_2(|x-y|),
    \end{align*}
    uniformly for all $(t\,,x),(s\,,y) \in I\times J$, where
    \begin{align*}
        &\phi_1(|t-s|) := |t-s|^{\theta_1} \quad \text{and}\\
        &\phi_2(|x-y|) := \begin{cases}
            |x-y|^{\theta_2}, & 0<\theta_2<1,\\
            |x-y|\sqrt{\log_+(1/|x-y|)},& \theta_2 = 1,\\
            |x-y|, & \theta_2>1.
        \end{cases}
    \end{align*}
    Moreover, $t \mapsto u(t\,,x)$ is a.s.~H\"older continuous on $I$ of order $\theta_1^-$; $x \mapsto u(t\,,x)$ is a.s.~H\"older continuous on $J$ of order $\theta_2^-$ when $0<\theta_2<1$, and a.s.~continuously differentiable when $\theta_2>1$.
\end{lemma}

Next, we discuss optimal regularities of the solution when $0<\theta_2 \le 1$.
The exact moduli of continuity for the solution is known when $0<\theta_2<1$ (see \cite{HSWX20, LX23}): In this case, for any fixed compact rectangle $I\times J \subset (0\,,\infty) \times \R^d$, the solution $u=\{u(t\,,x)\}_{(t,x)\in I\times J}$ satisfies SLND on $I\times J$, and for any fixed $t_0>0$, $x_0 \in \R^d$, there exist constants $K_1,K_2,K_3,K_4,K_5,K_6 \in (0\,,\infty)$ such that
\begin{gather*}
    \lim_{h\to0^+} \sup_{t,s\in I:0<|t-s|\le h}\frac{|u(t\,,x_0)-u(s\,,x_0)|}{|t-s|^{\theta_1}\sqrt{\log(1/|t-s|)}} = K_1 \quad \text{a.s.},\\
    \lim_{h\to0^+}\sup_{t\in I:0<|t-t_0|\le h} \frac{|u(t\,,x_0)-u(t_0\,,x_0)|}{|t-t_0|^{\theta_1}\sqrt{\log\log(1/|t-t_0|)}} = K_2 \quad \text{a.s.},\\
    \lim_{h\to0^+}\sup_{x,y\in J: 0<|x-y|\le h} \frac{|u(t_0\,,x)-u(t_0\,,y)|}{|x-y|^{\theta_2}\sqrt{\log(1/|x-y|)}} = K_3 \quad \text{a.s.},\\
    \lim_{h\to0^+} \sup_{x \in J: 0<|x-x_0|\le h} \frac{|u(t_0\,,x)-u(t_0\,,x_0)|}{|x-x_0|^{\theta_2}\sqrt{\log\log(1/|x-x_0|)}} = K_4 \quad \text{a.s.},\\
    \lim_{h\to0^+}\sup_{z,z'\in I\times J: 0<\phi(z-z')\le h} \frac{|u(z)-u(z')|}{\phi(z-z')\sqrt{\log(1/\phi(z-z'))}} = K_5 \quad \text{a.s.},\\
    \lim_{h\to0^+}\sup_{z\in I\times J: 0<\phi(z-z_0)\le h} \frac{|u(z)-u(z_0)|}{\phi(z-z_0)\sqrt{\log\log(1/\phi(z-z_0))}} = K_6 \quad \text{a.s.},
\end{gather*}
where $z_0 = (t_0\,,x_0)$.
The moduli of continuity in the last display can also be derived from our results thanks to the SLND property when $0<\theta_2<1$ (see Proposition \ref{prop:SLND}).

When $\theta_2=1$, whether $u$ satisfies SLND in $x$ or in $(t\,,x)$ is an open problem, and due to this, the exact uniform moduli of continuity in $x$ and in $(t\,,x)$ are also open \cite{HSWX20, H23, CLX}.
In what follows, we will tackle this problem in the particular case that $\alpha \in (0\,,2]$, $\ell=d=1$. 
In this case, 
\begin{align}\label{theta_2=1}
    \theta_2 = 1 \quad \Leftrightarrow \quad \alpha H=3/2.
\end{align}

\subsection{Localization}
From now on, we focus on the case that
\begin{align}\label{alphaH=3/2}
    \alpha \in (0\,,2], \quad \ell=d=1, \quad H \in [1/2\,,1), \quad \alpha H=3/2.
\end{align}
Note that, in this case, the ranges of $\alpha$ and $H$ become $\alpha \in (3/2\,,2]$ and $H \in (3/4\,,1)$.
Our goal is to solve the exact uniform modulus of continuity (Theorem \ref{th:um:SPDE}).
Due to the lack of SLND property for $\{ u(t_0\,,x) \}_{x \in J}$, we start by analyzing the localization property of the spatial increments of the solution \eqref{mild:sol}.

Fix $t>0$ and $x \in \R$.
For any $h>0$, consider the spatial increment 
\begin{align}\begin{split}\label{Delta:h:u}
    \Delta^h u &= \Delta^h u(t\,,x) := u(t\,,x+h)-u(t\,,x)\\
    &= \int_0^t \int_{\R^d} \left[ G(t-s\,,x+h-y) - G(t-s\,,x-y) \right] W(\d s\, \d y).
\end{split}\end{align}
Let $\delta \in (0\,,t)$, $\varepsilon>0$, and decompose the increment as $\Delta^h u = \Delta^h_A u + \Delta^h_B u$, 
where
\begin{align*}
    A = [t-\delta\,,t] \times [x-\varepsilon\,,x+\varepsilon], \quad 
    B = ([0\,,t] \times \R) \setminus A,
\end{align*}
and $\Delta^h_C u = \Delta^h u(t\,,x)$ denotes
\begin{align}\label{Delta:u}
    \Delta^h_C u(t\,,x) := \iint_{C} \left[ G(t-s\,,x+h-y) - G(t-s\,,x-y) \right] W(\d s\, \d y).
\end{align}
We wish to show that the increment localizes on region $A$ as $h\to 0^+$ under a suitable choice of $\varepsilon$ and $\delta$.
To this end, we will quantify the localization error 
\begin{align}\label{loc:error}
    \Delta^h_B u = \Delta^h u - \Delta^h_A u
\end{align}
with precise dependence on $\varepsilon$ and $\delta$.
In particular, we split the error into $\Delta^h_B u = \Delta^h_{B_1} u + \Delta^h_{B_2} u$, where $B = B_1 \cup B_2$ with
\begin{align*}
    B_1 = [0\,,t-\delta)\times \R \quad \text{and} \quad
    B_2 = [t-\delta\,,t] \times \{ y \in \R : |x-y| > \varepsilon\},
\end{align*}
and estimate the two parts in turn.

\begin{lemma}
Suppose $\alpha >0$, $H\in (1/2\,,1)$, $\ell=d=1$, and $\alpha H = 3/2$.
Fix $0<t_1<t_2$. Then there is $K>0$ such that
\[
    \|\Delta^h_{B_1} u(t\,,x)\|_2^2 \le K\left[ h^2 \log(1/\delta^{1/\alpha}) + h^2 + \delta^{2/\alpha} e^{-\delta/h^\alpha}\right]
\]
uniformly for all $t \in [t_1\,,t_2]$, $x\in \R$, $\delta \in (0\,,1\wedge t_1)$ and $h \in (0\,, \delta^{1/\alpha})$.
\end{lemma}

\begin{proof}
    First, we use \eqref{W:cov}, \eqref{Delta:u} and $L^2$-isometry of Wiener integrals, and the change of variables $r_i = t-s_i$, $z = x-y$ to deduce that $\|\Delta^h_{B_1} u\|_2^2$ equals
    \begin{align*}
        &\!\int_0^{t-\delta} \d s_1 \int_0^{t-\delta} \d s_2 |s_1-s_2|^{2H-2}\int_\R \d y \, [G(t-s_1\,, x+h-y)-G(t-s_1\,, x-y)]\\
        & \hskip2.2in \times [G(t-s_2\,, x+h-y)-G(t-s_2\,, x-y)]\\
        &\!\!= \int_\delta^t \d r_1 \int_\delta^t \d r_2  |r_1-r_2|^{2H-2}\! \int_{\R} \d z [G(r_1\,,z-h)-G(r_1\,,z)] [G(r_2\,,z-h)-G(r_2\,,z)].
    \end{align*}
    The Fourier transform of $\Lambda_{1,\nu}$ is $\hat\Lambda_{1,\nu} = C_\nu\, \Lambda_{1,1-\nu}$ for some explicit constant $C_\nu$; see \cite[(12.10)]{Mattila}.
    By \eqref{G:hat}, the Fourier transform of $g(r\,,z) := G(r\,,z){\bf 1}_{[\delta,t]}(r)$ with respect to variable $(r\,,z)$ is
    \[
        \hat{g}(\tau\,,\xi) = \frac{e^{-\delta(i\tau+|\xi|^\alpha)} - e^{-t(i\tau+|\xi|^\alpha)}}{i\tau - |\xi|^\alpha} \qquad \forall (\tau\,,\xi) \in \R\times \R.
    \]
    Hence, by Plancherel's theorem, the change of variable $\xi \to \delta^{-1/\alpha}\xi$, and $2H=3/\alpha$,
    \begin{align*}
         &\|\Delta^h_{B_1} u\|_2^2
         \propto \int_\R \d \tau \int_\R \d \xi \, |e^{-ih\xi}-1|^2 |\tau|^{1-2H}e^{-2|\xi|^\alpha \delta} \frac{|1-e^{-(i\tau+|\xi|^\alpha)(t-\delta)}|^2}{|\tau|^2+|\xi|^{2\alpha}}\\
         &= \delta^{-1/\alpha+2H} \int_\R \d \tau \int_\R \d \xi\, |e^{-ih \delta^{-1/\alpha}\xi}-1|^2 |\tau|^{1-2H} e^{-2|\xi|^\alpha }\frac{|1-e^{-(i\tau+|\xi|^\alpha)\frac{t-\delta}{\delta}}|^2}{|\tau|^2+|\xi|^{2\alpha}}\\
         &= \delta^{2/\alpha} (I_1 + I_2 + I_3 + I_4 + I_5 + I_6),
    \end{align*}
    where $I_k$ corresponds to the integral over the region ($k$) specified as follows:
    \begin{enumerate}
        \item $|\xi|\le 1$ and $|\tau|\le |\xi|^\alpha$;
        \item $|\xi|\le 1$ and $|\tau| > |\xi|^\alpha$;
        \item $1<|\xi| \le h^{-1}\delta^{1/\alpha}$ and $|\tau| \le |\xi|^\alpha$;
        \item $1<|\xi| \le h^{-1}\delta^{1/\alpha}$ and $|\tau|>|\xi|^\alpha$;
        \item $|\xi|>h^{-1}\delta^{1/\alpha}$ and $|\tau| \le |\xi|^\alpha$;
        \item $|\xi|>h^{-1}\delta^{1/\alpha}$ and $|\tau|> |\xi|^\alpha$.
    \end{enumerate}
    We will use the elementary bounds
    \begin{align}\label{exp:bd:1}
        |e^{-ih \delta^{-1/\alpha}\xi}-1|^2 &\lesssim 1 \wedge (h^2\delta^{-2/\alpha}|\xi|^2),\\
        \frac{|1-e^{-(i\tau+|\xi|^\alpha)\frac{t-\delta}{\delta}}|^2}{|\tau|^2+|\xi|^{2\alpha}} &\lesssim \delta^{-2} \wedge |\tau|^{-2} \wedge |\xi|^{-2\alpha},
        \label{exp:bd:2}
    \end{align}
    valid uniformly for all $\tau, \xi \in \R$, $t\in [t_1\,,t_2]$, $\delta \in (0\,, 1\wedge t_1)$ and $h\in (0\,,\delta^{1/\alpha})$. 
    In particular, by \eqref{exp:bd:1}, \eqref{exp:bd:2} and $2H\alpha = 3$,
    \begin{align*}
        I_1 &\lesssim \int_0^{\delta^{1/\alpha}} \d \xi \int_0^{|\xi|^\alpha} \d \tau \, h^2 \delta^{-2/\alpha} |\xi|^2 |\tau|^{1-2H} e^{-2|\xi|^\alpha} \delta^{-2}\\
        & \quad + \int_{\delta^{1/\alpha}}^1 \d \xi \int_0^{|\xi|^\alpha} \d \tau \, h^2 \delta^{-2/\alpha}|\xi|^2 |\tau|^{1-2H} e^{-2|\xi|^\alpha} |\xi|^{-2\alpha}\\
        & \lesssim h^2 \delta^{-2-2/\alpha} \int_0^{\delta^{1/\alpha}} \d \xi \,|\xi|^{2+2\alpha-2H\alpha} + h^2 \delta^{-2/\alpha} \int_{\delta^{1/\alpha}}^1 \d \xi \, |\xi|^{2-2H\alpha}\\
        & \lesssim h^2 \delta^{-2/\alpha} + h^2 \delta^{-2/\alpha} \log(1/\delta^{1/\alpha}) \lesssim h^2 \delta^{-2/\alpha} \log(1/\delta^{1/\alpha}).
    \end{align*}
    For $I_2$, we use Fubini's theorem to write $\int_0^1 \d \xi \int_{|\xi|^\alpha}^\infty \d \tau (\cdots) = \int_0^\infty \d \tau \, \int_0^{|\tau|^{1/\alpha} \wedge 1} \d \xi (\cdots)$ and then use \eqref{exp:bd:1}, \eqref{exp:bd:2}, and $3/\alpha=2H$ to deduce that
    \begin{align*}
        I_2 
        & \lesssim \int_0^\delta \d \tau \int_0^{|\tau|^{1/\alpha}\wedge 1} \d \xi \, h^2 \delta^{-2/\alpha} |\xi|^2 |\tau|^{1-2H} e^{-2|\xi|^\alpha} \delta^{-2}\\
        & \quad + \int_\delta^\infty \d \tau \int_0^{|\tau|^{1/\alpha}\wedge 1} \d \xi \, h^2 \delta^{-2/\alpha} |\xi|^2 |\tau|^{1-2H} e^{-2|\xi|^\alpha} |\tau|^{-2}\\
        & \lesssim h^2 \delta^{-2-2/\alpha} \int_0^\delta \d \tau \, |\tau|^{1-2H} \int_0^{|\tau|^{1/\alpha}} \d \xi \, |\xi|^2 + h^2 \delta^{-2/\alpha} \int_0^1 \d \xi \int_{\delta \vee|\xi|^\alpha}^\infty \d \tau \, |\tau|^{-1-2H}\\
        & \lesssim h^2 \delta^{-2-2/\alpha} \int_0^\delta \d \tau \, |\tau| + h^2 \delta^{-2/\alpha} \int_0^1 \d \xi \, |\xi|^2 (\delta \vee|\xi|^\alpha)^{-2H}\\
        & \lesssim h^2 \delta^{-2/\alpha} + h^2 \delta^{-2/\alpha} \left[ \int_0^{\delta^{1/\alpha}} \d \xi \, |\xi|^2 \delta^{-2H} + \int_{\delta^{1/\alpha}}^1 \d \xi \, |\xi|^{2-2H\alpha}\right]\\
        & \lesssim h^2 \delta^{-2/\alpha} + h^2 \delta^{-2/\alpha} \left[ \delta^{3/\alpha-2H} + \log(1/\delta^{1/\alpha})\right] \lesssim h^2 \delta^{-2/\alpha} \log(1/\delta^{1/\alpha}).
    \end{align*}
    Next, by \eqref{exp:bd:1}, \eqref{exp:bd:2} and $2H\alpha = 3$,
    \begin{align*}
        I_3 &\lesssim \int_1^{h^{-1}\delta^{1/\alpha}} \d \xi \int_0^{|\xi|^\alpha} \d \tau \, h^2 \delta^{-2/\alpha} |\xi|^2 |\tau|^{1-2H} e^{-2|\xi|^\alpha} |\xi|^{-2\alpha}\\
        & \lesssim h^2 \delta^{-2/\alpha} \int_1^{h^{-1}\delta^{1/\alpha}} \d \xi \, e^{-2|\xi|^\alpha} |\xi|^{2-2H\alpha}\\
        & \le h^2 \delta^{-2/\alpha} \int_1^\infty \d \xi \, e^{-2|\xi|^\alpha} \lesssim h^2\delta^{-2/\alpha}.
    \end{align*}
    Similarly, we have
    \begin{align*}
        I_4 &\lesssim \int_1^{h^{-1}\delta^{1/\alpha}} \d \xi \int_{|\xi|^\alpha}^\infty \d \tau \, h^2 \delta^{-2/\alpha} |\xi|^2 |\tau|^{1-2H} e^{-2|\xi|^\alpha} |\tau|^{-2}\\
        & \lesssim h^2 \delta^{-2/\alpha} \int_1^{h^{-1}\delta^{1/\alpha}} \d \xi \, e^{-2|\xi|^\alpha} |\xi|^{2-2H\alpha} 
        \lesssim h^2 \delta^{-2/\alpha}.
    \end{align*}
    For $I_5$, we use \eqref{exp:bd:1}, \eqref{exp:bd:2}, $2H\alpha=3$ and $h^{-1}\delta^{1/\alpha} \ge 1$ to deduce that
    \begin{align*}
        I_5 &\lesssim \int_{h^{-1}\delta^{1/\alpha}}^\infty \d \xi \int_0^{|\xi|^\alpha} \d \tau \, |\tau|^{1-2H} e^{-2|\xi|^\alpha} |\xi|^{-2\alpha}\\
        & \lesssim \int_{h^{-1}\delta^{1/\alpha}}^\infty \d \xi \, e^{-2|\xi|^\alpha} |\xi|^{-2H\alpha}
        \le \int_{h^{-1}\delta^{1/\alpha}}^\infty \d \xi \, e^{-2|\xi|^\alpha}\\
        & \le e^{- (h^{-1}\delta^{1/\alpha})^\alpha} \int_0^\infty \d \xi \, e^{- |\xi|^\alpha} \lesssim e^{-\delta/h^\alpha}.
    \end{align*}
    Finally, we estimate $I_6$ in a similar way:
    \begin{align*}
        I_6 &\lesssim \int_{h^{-1}\delta^{1/\alpha}}^\infty \d \xi \int_{|\xi|^\alpha}^\infty \d \tau \, |\tau|^{1-2H} e^{-2|\xi|^\alpha} |\tau|^{-2}\\
        & \lesssim \int_{h^{-1}\delta^{1/\alpha}}^\infty \d \xi \, e^{-2|\xi|^\alpha} |\xi|^{-2H\alpha}
        \le \int_{h^{-1}\delta^{1/\alpha}}^\infty \d \xi \, e^{-2|\xi|^\alpha} \lesssim e^{-\delta/h^\alpha}.
    \end{align*}
    Combine all these estimates to finish the proof.
\end{proof}

\begin{lemma}
Suppose $\alpha \in (0\,,2]$, $H \in (1/2\,,1)$, and $\ell=d=1$.
Fix $0<t_1<t_2$. Then there is $K>0$ such that
\[
    \|\Delta^h_{B_2} u(t\,,x)\|_2^2 \le K h^2 \varepsilon^{-3-2\alpha} \delta^{2H+2}
\]
uniformly for all $t \in [t_1\,,t_2]$, $x \in \R$, $\delta \in (0\,,1 \wedge t_1)$, $\varepsilon>0$, $h \in [0\,,\delta^{1/\alpha} \wedge \tfrac{\varepsilon}{2})$.
\end{lemma}

\begin{proof}
    By \eqref{W:cov}, \eqref{Delta:u}, $L^2$-isometry of Wiener integrals, and the change of variables $r_i=t-s_i$, $z=x-y$, we see that $\|\Delta^h_{B_2} u\|_2^2$ equals
    \begin{align*}
        \int_0^\delta \d r_1 \int_0^\delta \d r_2 |r_1-r_2|^{2H-2} \int_{|z|>\varepsilon} \d y [G(r_1\,,z+h)-G(r_1\,,z)][G(r_2\,,z+h)-G(r_2\,,z)].
    \end{align*}
    When $\alpha=2$, it follows from $\sup_{z>0} z^5 e^{-z^2}<\infty$ that
    \[
        |\partial_y G(s\,,y)| = \frac{|y|/(2s)}{\sqrt{4\pi s}} \exp\left(-\frac{y^2}{4s}\right) \lesssim s |y|^{-4}
        \quad \forall s>0, y \in \R.
    \]
    When $\alpha \in (0\,,2)$, thanks to a gradient estimate \cite[Lemma 2.2]{CZ16}, we have
    \begin{align}\label{gradient:G}
        |\partial_y G(s\,,y)| \lesssim s |y|^{-2-\alpha} \quad \forall s > 0, y \in \R.
    \end{align}
    Thus \eqref{gradient:G} holds for $\alpha \in (0\,,2]$.
    Also, for $0 \le h < \varepsilon/2$ and $|z|>\varepsilon$ imply that
    \begin{align*}
        |z+h| \ge |z|-h > |z|-\varepsilon/2 > |z|-|z|/2 = |z|/2. 
    \end{align*}
    By the mean value theorem, for any $h \in (0\,,\varepsilon/2)$ and $|z|>\varepsilon$, there exists $h^* \in [0\,,h]$ such that
    \begin{align*}
        |G(r\,,z+h)-G(r\,,z)|
        \lesssim h r |z+h^*|^{-2-\alpha} \lesssim h r |z|^{-2-\alpha}.
    \end{align*}
    It follows that
    \begin{align*}
        \|\Delta^h_B u\|_2^2
        \lesssim \int_0^\delta \d r_1 \int_0^\delta \d r_2 |r_1-r_2|^{2H-2} \int_{|z|>\varepsilon} \frac{h^2 r_1r_2}{|z|^{-4-2\alpha}}.
    \end{align*}
    The change of variables $r_i = \delta s_i$ then yields the desired estimate.  
\end{proof}

Combining the above two lemmas, we obtain the following estimate for the localization error $\Delta^h_B u = \Delta^h_B u (t\,,x)$ introduced in \eqref{loc:error}.

\begin{proposition}\label{prop:loc}
    Suppose \eqref{alphaH=3/2}, and set
    \begin{align*}
        \delta = h^{\rho\alpha} \quad \text{and} \quad
        \varepsilon= \delta^{\frac{2H+2}{3+2\alpha}}
        = h^{\rho}.
    \end{align*}
    Fix $0<t_1<t_2$. Then, there exists $K > 0$ such that for any $\rho \in (0\,,1)$, there exists $r_\rho \in (0\,,1)$ such that
    \begin{align*}
        \|\Delta^h_B u(t\,,x)\|_2^2 \le K \rho\, h^2 \log(1/h)
    \end{align*}
    uniformly for all $t \in [t_1\,,t_2]$, $x \in \R$ and $h \in (0\,,r_\rho]$.
\end{proposition}

\subsection{Proof of Theorem \ref{th:um:SPDE}}

We first prove \eqref{um:SPDE:x}.
By Lemma \ref{lem:SPDE:reg}, since $\theta_2=1$, there exist $C_1 \in (0\,,1)$ and $C_2>1$ such that
\begin{align}\label{u-u:x}
    C_1 \varphi_2(|x-y|) \le \|u(t_0\,,x)-u(t_0\,,y)\|_2 \le C_2 \varphi_2(|x-y|) \quad \forall x,y \in J.
\end{align}
where 
\[
    \varphi_2(r) = r \sqrt{\log_+(1/r)}.
\]
Clearly, $\varphi_2$ satisfies \eqref{um:phi}.
We now use our localization estimate to verify condition \eqref{um:cv}.
Recall the notation introduced in \eqref{Delta:h:u} and \eqref{Delta:u}.
Let $\rho\in (0\,,1)$.
For any $x,\tilde{x}\in J$ and $h, \tilde{h}>0$, we set 
\begin{align*}
    &A = [t_0-h^{\rho\alpha}\,, t_0] \times [x- h^{\rho}\,, x+  h^{\rho}], \quad B = ([0\,,t_0] \times \R)\setminus A,\\
    &\tilde{A} = [t_0-\tilde{h}^{\rho\alpha}\,, t_0] \times [\tilde{x}- \tilde{h}^{\rho}\,, \tilde{x}+ \tilde{h}^{\rho}], \quad \tilde{B} = ([0\,,t_0] \times \R)\setminus \tilde{A}.
\end{align*}
It follows that for any $r>0$, if $h\le r$, $\tilde{h} \le r$ and $|x-\tilde{x}| \ge 2r^{\rho}$, then $A \cap \tilde{A}=\varnothing$ and
\begin{align*}
    &\E[\Delta^h u(t_0\,,x) \Delta^{\tilde{h}} u(t_0\,,\tilde{x})]\\
    &= \E[(\Delta^h_A u(t_0\,,x) + \Delta^h_B u(t_0\,,x)) (\Delta^{\tilde{h}}_{\tilde{A}} u(t_0\,,\tilde{x}) + \Delta^{\tilde{h}}_{\tilde{B}}u(t_0\,,\tilde{x}))]\\
    &= 0 + \E[\Delta^{h}_{A} u(t_0\,,x) \Delta^{\tilde{h}}_{\tilde{B}} u(t_0\,,\tilde{x})] \\
    & \quad + \E[\Delta^{h}_{B} u(t_0\,,x) \Delta^{\tilde{h}}_{\tilde{A}} u(t_0\,,\tilde{x})] + \E[\Delta^{h}_{B} u(t_0\,,x) \Delta^{\tilde{h}}_{\tilde{B}} u(t_0\,,\tilde{x})],
\end{align*}
where in the last line, we have used $A \cap \tilde{A}=\varnothing$, which implies that $\Delta^h_A u(t_0\,,x)$ and $\Delta^{\tilde{h}}_{\tilde{A}} u(t_0\,,x)$ are independent.
Then, by the Cauchy-Schwarz inequality and the trivial bound $\|\Delta^h_A u(t_0\,,x)\|_2 \le \|\Delta^h u(t_0\,,x)\|_2$ (see \eqref{Delta:u}), we have
\begin{align*}
    |\E[\Delta^h u(t_0\,,x) &\Delta^{\tilde{h}} u(t_0\,,\tilde{x})]|
    \le \|\Delta^{h} u(t_0\,,x)\|_2 \|\Delta^{\tilde{h}}_{\tilde{B}} u(t_0\,,\tilde{x})\|_2\\
    &  + \|\Delta^{h}_{B} u(t_0\,,x)\|_2 \|\Delta^{\tilde{h}} u(t_0\,,\tilde{x})\|_2 + \|\Delta^{h}_{B} u(t_0\,,x)\|_2\|\Delta^{\tilde{h}}_{\tilde{B}} u(t_0\,,\tilde{x})\|_2.
\end{align*}
Thanks to Proposition \ref{prop:loc} and the upper bound in \eqref{u-u:x}, there exists $K_0>1$ such that for any $\rho \in (0\,,1)$, we can find $r_\rho \in (0\,,1)$ such that if $r \in (0\,,r_\rho]$, $h,\tilde{h} \in (0\,,r]$ and $|x-\tilde{x}|\ge 2r^\rho$, then
\begin{align*}
     &|\E[\Delta^h u(t_0\,,x) \Delta^{\tilde{h}} u(t_0\,,\tilde{x})]| \\
     & \le C_2 K_0 \sqrt{\rho}\, \varphi_2(h) \varphi_2(\tilde{h}) + C_2 K_0 \sqrt{\rho}\, \varphi_2(h) \varphi_2(\tilde{h}) + K_0^2 \rho \,\varphi_2(h) \varphi_2(\tilde{h})\\
     & \le (2C_2K_0 + K_0^2)\sqrt{\rho} \,\varphi_2(h) \varphi_2(\tilde{h}).
\end{align*}
For any $C_0 \in (0\,,1)$, 
we may choose and fix $\rho \in (0\,,1)$ such that 
\[
    \sqrt{\rho} \le \sqrt{1-C_0} \,C_1^2/(2C_2K_0+K_0^2) .
\]
Then, it follows from the lower bound in \eqref{u-u:x} that
\begin{align*}
     |\E[\Delta^h u(t_0\,,x) \Delta^{\tilde{h}} u(t_0\,,\tilde{x})]| \le \sqrt{1-C_0}\, \|\Delta^h u(t_0\,,x)\|_2 \|\Delta^{\tilde{h}} u(t_0\,,\tilde{x})\|_2
\end{align*}
uniformly for all $r \in (0\,,r_\rho]$, $h,\tilde{h} \in (0\,,r]$ and $x,\tilde{x} \in J$ with $|x-\tilde{x}| \ge 2r^\rho$.
This verifies condition \eqref{um:cv}.
Therefore, Theorem \ref{th:um} implies that \eqref{um:SPDE:x} holds for some constant $0<K_3<\infty$.

For \eqref{um:SPDE:t}, Lemma \ref{lem:SPDE:reg} implies that
\[
    \|u(t\,,x_0)-u(s\,,x_0)\|_2 \asymp \varphi_1(|t-s|) := |t-s|^{1/\alpha} \quad \forall t,s\in I,
\]
and $\varphi_1$ satisfies \eqref{um:phi}.
By \cite[Theorem 1.3]{CLX}, $\{u(t\,,x_0)\}_{t \in I}$ satisfies SLND on $I$, so condition \eqref{um:cv} holds with $\rho=1$.
Hence, \eqref{um:SPDE:t} follows from Theorem \ref{th:um}.

Next, we turn to the proof of  \eqref{um:SPDE:tx}.
Set 
\[ 
    \phi((t\,,x)\,,(s\,,y)) = \phi_1(|t-s|)\vee \phi_2(|x-y|),
\]
where $\phi_1(r) = r^{1/\alpha}$ and $\phi_2(r) = r \sqrt{\log_+(1/r)}$.
Lemma \ref{lem:SPDE:reg} implies that
\begin{align}\label{u-u:tx}
    \|u(t\,,x)-u(s\,,y)\|_2 \asymp \phi((t\,,x)\,,(s\,,y)) \quad \forall (t\,,x),(s\,,y) \in I\times J.
\end{align}
By Lemma \ref{lem:01}(i), \eqref{um:SPDE:tx} holds for some constant $K_5 \in [0\,,\infty]$.
It remains to show that $0<K_5<\infty$.
In fact, by \eqref{u-u:tx} and \eqref{um:SPDE:t},
\begin{align*}
    K_5=&\lim_{h\to0^+} \sup_{z,z'\in I\times J: 0<\phi(z,z') \le h} \frac{|u(z)-u(z')|}{\phi(z\,,z')\sqrt{\log(1/\phi(z\,,z'))}}\\
    &\ge \lim_{h\to0^+} \sup_{t,s\in I: 0<|t-s| \le h} \frac{|u(t\,,x_0)-u(s\,,x_0)|}{|t-s|^{1/\alpha}\sqrt{\log(1/|t-s|^{1/\alpha})}} = \alpha K_4>0.
\end{align*}
To show that $K_5 < \infty$, note that we have from definition \eqref{psi} that
\begin{align}\label{SPDE:psi}
    \psi(r) = \phi_1^{-1}(r)\phi_2^{-1}(r) \sim r^{\alpha + 1}(\log_+(1/r))^{-1/2} \quad \text{as $r \to 0^+$}
\end{align}
and
\begin{align}\label{SPDE:psi:inv}
    \psi^{-1}(r) \asymp r^{1/(\alpha+1)} (\log_+(1/r))^{1/(2\alpha+2)} \quad \text{as $r \to 0^+$.}
\end{align}
Then, as $h \to 0^+$,
\begin{align*}
    \int_0^h \frac{\psi^{-1}(r) \, \d r}{r \sqrt{\log(1/r)}}
    &\lesssim h^{1/(\alpha+1)} (\log(1/h))^{1/(2\alpha+2)-1/2} \\
    &\le h^{1/(\alpha+1)} (\log(1/h))^{1/(2\alpha+2)+1/2} \lesssim \psi^{-1}(h) \sqrt{\log(1/h)} = o(1).
\end{align*}
Hence, Assumptions \ref{a:sigma:aniso} and \ref{a:psi:aniso} are satisfied.
Theorem \ref{th:um:aniso} implies that
\[
    \lim_{h\to0^+} \sup_{z,z'\in I\times J: 0<\phi(z,z') \le h} \frac{|u(z)-u(z')|}{\phi(z\,,z')\sqrt{\log(1/\phi(z\,,z'))}} < \infty \quad \text{a.s.}
\]
This proves $K_5<\infty$, and hence \eqref{um:SPDE:tx}.
The proof of Theorem \ref{th:um:SPDE} is complete. \qed\\

The proof yields the following decorrelation result for the spatial increments.

\begin{proposition}\label{prop:corr}
	Suppose \eqref{alphaH=3/2}, and fix $0<t_1<t_2$.
	Then, for any $\varepsilon \in (0\,,1)$, there exist $\rho \in (0\,,1)$ such that $\rho = O(\varepsilon^2)$ and $r_\rho \in (0\,,1)$ such that
	\begin{align*}
		|\E[(u(t\,,x)-u(t\,,y))&(u(t\,,x') - u(t\,,y'))]| \\
		&\le \varepsilon \|u(t\,,x)-u(t\,,y)\|_2 \|u(t\,,x') - u(t\,,y')\|_2
	\end{align*}
	uniformly for all $t\in [t_1\,,t_2]$, $r \in (0\,,r_\rho]$ and $x,y,x',y' \in \R$ with $|x-y| \le r$, $|x'-y'| \le r$ and $|x-x'| \ge 2r^\rho$.
\end{proposition}

\subsection{Proof of Theorem \ref{th:lm:SPDE}}

We start with the proof of \eqref{lm:SPDE:x}.
Lemma \ref{lem:01}(ii) and \eqref{u-u:x} imply that \eqref{lm:SPDE:x} holds for some constant $K_6 \in [0\,,\infty]$.
Next, we show that $K_6 > 0$.
By \eqref{G:hat}, for any $t>0$ and $x \in \R$, the space-time Fourier transform of $g_{t,x}(r\,,z) = G(t-r\,,x-z) {\bf 1}_{[0,t](r)}$ is
\begin{align}\label{g:hat}
    \hat{g}_{t,x}(\tau\,,\xi) = e^{-i\xi x}\frac{e^{-i\tau t}-e^{-t|\xi|^\alpha}}{i\tau - |\xi|^\alpha} \quad \forall (\tau\,,\xi) \in \R \times \R.
\end{align}
Let $\tilde{W}=\tilde{W}_1 + i\tilde{W}_2$, where $\tilde{W}_1$ and $\tilde{W}_2$ are independent white noises on $\R \times \R$.
For any $t>0$ and $x \in \R$, define
\begin{align}\label{v}
    v(t\,,x) = \Re \int_\R \int_\R \hat{g}_{t,x}(\tau\,,\xi) |\tau|^{(1-2H)/2} \tilde{W}(\d \tau\, \d \xi).
\end{align}
The Gaussian random field $\{ v(t\,,x) \}_{t > 0, x \in \R}$ is known as the harmonizable representation of the solution \cite{Balan, DMX17, LX23} and has the same law as $\{ u(t\,,x) \}_{t>0, x \in \R}$ because they both have zero mean and the same covariance function by Parseval's identity.
For any Borel set $A \subset [0\,,\infty)$, define the truncated random field
\[
    v(A\,,t\,,x) = \Re \iint_{\sqrt{|\tau|^{2\theta_1}+ |\xi|^{2\theta_2}} \in A} \hat{g}_{t,x}(\tau\,,\xi) |\tau|^{(1-2H)/2} \tilde{W}(\d \tau\, \d \xi).
\]
If $A \cap B = \varnothing$, then $\{v(A\,,t\,,x)\}_{t>0,x\in \R}$ and $\{v(B\,,t\,,x)\}_{t>0,x\in \R}$ are independent.
\begin{lemma}\label{lem:v}
    Fix $0<t_1<t_2$.
    Then, under \eqref{alphaH=3/2}, 
    \[
        \|v([a\,,b]^c\,,t\,,x+h) - v([a\,,b]^c\,,t\,,x)\|_2 \lesssim h\sqrt{\log a} + b^{-1}
    \]
    uniformly for all $2 \le a < b \le \infty$, $t \in [t_1\,,t_2]$, $x \in \R$ and $h \in (0\,,1]$.
\end{lemma}

\begin{proof}
    Using \eqref{g:hat} and the elementary inequalities $|1-e^{-i\xi h}|^2 \lesssim 1 \wedge |h\xi|^2$ and
    \[
        \frac{|e^{-i\tau t}-e^{-t|\xi|^\alpha}|^2}{|\tau|^2+|\xi|^{2\alpha}}\lesssim t^2 \wedge |\tau|^{-2} \wedge |\xi|^{-2\alpha},
    \]
    we deduce that for $2 \le a < b \le \infty$, $t \in [t_1\,,t_2]$, $x \in \R$ and $h \in (0\,,1]$,
    \begin{align*}
        &\|v([a\,,b]^c\,,t\,,x+h) - v([a\,,b]^c\,,t\,,x)\|_2^2\\
        & = \iint_{\sqrt{|\tau|^{2/\alpha}+ |\xi|^{2}} \in [0,a] \cup [b, \infty)}  \frac{|1-e^{-i\xi h}|^2|e^{-i\tau t}-e^{-t|\xi|^\alpha}|^2}{|\tau|^2+|\xi|^{2\alpha}} |\tau|^{1-2H} \d \tau \, \d \xi\\
        & \lesssim \iint_{\sqrt{|\tau|^{2/\alpha}+ |\xi|^{2}} \le 1} h^2 |\xi|^2 |\tau|^{1-2H} \d \tau \, \d \xi + \iint_{1 \le \sqrt{|\tau|^{2/\alpha}+ |\xi|^{2}} \le a} h^2 \frac{|\xi|^2 |\tau|^{1-2H}}{|\tau|^2+|\xi|^{2\alpha}} \d \tau \, \d \xi\\
        & \quad + \iint_{\sqrt{|\tau|^{2/\alpha}+ |\xi|^{2}} \ge b} \frac{|\tau|^{1-2H}}{|\tau|^2+|\xi|^{2\alpha}} \d \tau \, \d \xi\\
        & \lesssim h^2 \int_0^1 \d \tau \, |\tau|^{1-2H} \int_0^1 \d \xi \, |\xi|^2 
        + h^2 \iint_{1 \le \sqrt{|s|^2 + |\xi|^2} \le a} \frac{|\xi|^2|s|^{\alpha(1-2H)}}{|s|^{2\alpha}+|\xi|^{2\alpha}} \, |s|^{\alpha-1} \d s \, \d \xi\\
        & \quad + \iint_{\sqrt{|\tau|^{2/\alpha}+ |\xi|^{2}} \ge b, |\tau|^{1/\alpha}\le |\xi|} \frac{|\tau|^{1-2H}\d \tau \, \d \xi }{|\tau|^2+|\xi|^{2\alpha}} +\iint_{\sqrt{|\tau|^{2/\alpha}+ |\xi|^{2}} \ge b, |\tau|^{1/\alpha}\ge |\xi|} \frac{|\tau|^{1-2H}\d \tau \, \d \xi}{|\tau|^2+|\xi|^{2\alpha}}.
    \end{align*}
    Then, by $2\alpha H = 3$, the change of variable $z = (s\,,\xi)$ and polar coordinates $|z| = r$,
    \begin{align*}
        \| v&([a\,,b]^c\,,t\,,x+h) - v([a\,,b]^c\,,t\,,x)\|_2^2
        \lesssim h^2 + h^2 \iint_{1\le |z| \le a} |z|^{-2}\, \d z\\
        & \quad + \int_b^\infty \d \xi \int_0^{|\xi|^\alpha} \d \tau \, |\tau|^{1-2H} |\xi|^{-2\alpha} + \int_{b^\alpha}^\infty \d \tau \int_0^{|\tau|^{1/\alpha}} \d \xi \, |\tau|^{-1-2H}\\
        & \lesssim h^2 + h^2 \int_1^a r^{-1} \d r + \int_b^\infty \d \xi \, |\xi|^{-2\alpha H} + \int_{b^\alpha}^\infty \d \tau \, |\tau|^{-1-2H+1/\alpha}\\
        & \lesssim h^2 + h^2 \log a + b^{1-2\alpha H} + b^{\alpha(-2H+1/\alpha)}\\
        & \lesssim h^2 \log a + b^{-2}.
    \end{align*}
    This completes the proof of Lemma \ref{lem:v}.
\end{proof}

Fix $t_0>0$ and $x_0 \in \R$.
Fix two numbers $\rho\in (0\,,1)$ and $\gamma>1$ whose values will be determined later.
For each $n \in \N$, let $x_n = x_0 + h_n$, where 
\[
    h_n = \exp(-e^{n^\gamma}).
\]
To simplify notation, write
\begin{align*}
    &\Delta^h u := u(t_0\,,x_0+h) - u(t_0\,,x_0),\\
    &\Delta^h v(A) := v(A\,,t_0\,,x_0+h) - v(A\,,t_0\,,x_0).
\end{align*}
Then, for any $m > n \ge 1$ and $2 \le a \le b \le c \le d \le \infty$, we have
\begin{align*}
    \E[\Delta^{h_n} u \, \Delta^{h_m} u]
    &= \E[(\Delta^{h_n} v([a\,,b])+ \Delta^{h_n} v([a\,,b]^c))(\Delta^{h_m}v([c\,,d])+\Delta^{h_m}v([c\,,d]^c))]\\
    &= 0+ \E[\Delta^{h_n} v([a\,,b])\Delta^{h_m}v([c\,,d]^c)] \\
    & \quad + \E[\Delta^{h_n} v([a\,,b]^c)\Delta^{h_m}v([c\,,d])] + \E[\Delta^{h_n} v([a\,,b]^c)\Delta^{h_m}v([c\,,d]^c)],
\end{align*}
where we have used the independence of $\Delta^{h_n} v([a\,,b]\,,x_0)$ and $\Delta^{h_m} v([c\,,d]\,,x_0)$.
By the Cauchy-Schwarz inequality and the bound $\|\Delta^h v(A)\|_2 \le \|\Delta^h v([0\,,\infty))\|_2 = \|\Delta^h u\|_2$, we have
\begin{align*}
    |\E[\Delta^{h_n} u \, \Delta^{h_m} &u]| \le \|\Delta^{h_n}u\|_2 \|\Delta^{h_m}v([c\,,d]^c)\|_2 \\
    &  + \|\Delta^{h_n} v([a\,,b]^c)\|_2 \|\Delta^{h_m} u\|_2 + \|\Delta^{h_n} v([a\,,b]^c)\|_2 \|\Delta^{h_m}v([c\,,d]^c)\|_2.
\end{align*}
Set
\[
    a = \exp(\rho e^{n^\gamma}), \quad
    b = \exp(e^{n^\gamma}), \quad
    c = \exp(\rho e^{m^\gamma}),\quad
    d= \infty.
\]
Now choose $\gamma = \gamma(\rho) = \log(1/\rho)>1$. 
This implies that for all $m>n\ge 1$,
\[
    e^{m^\gamma - n^\gamma} \ge e^{(n+1)^\gamma-n^\gamma} \ge e^{\gamma n^{\gamma-1}} \ge e^\gamma = \rho^{-1},
\]
and hence $b<c$.
Then, by Lemma \ref{lem:v} and \eqref{u-u:x}, there exists $n_0 \ge 1$ such that
\begin{align*}
    |\E[\Delta^{h_n} u \, \Delta^{h_m}u]|
    \lesssim \sqrt{\rho}\, \phi_2(h_n) \phi_2(h_m) \quad \forall m>n \ge n_0,
\end{align*}
and hence for any fixed $C_0\in (0\,,1)$, it is possible to choose $\rho \ll 1$ such that
\begin{align}\label{cov:Delta:u}
    |\E[\Delta^{h_n} u \, \Delta^{h_m}u]|
    \lesssim \sqrt{1-C_0}\, \|\Delta^{h_n} u\|_2 \|\Delta^{h_m}u\|_2 \quad \forall m>n \ge n_0.
\end{align}
Let $\xi_0,\xi_1,\dots$ be i.i.d.~standard Gaussian random variables.
For each $n\in \N$, define
\[
    Y_n = \frac{u(t_0\,,x_n)-u(t_0\,,x_0)}{\|u(t_0\,,x_n)-u(t_0\,,x_0)\|_2},\quad
    Z_n = (1-C_0)^{1/4} \xi_0 + (1-\sqrt{1-C_0})^{1/2} \xi_n.
\]
Then $\E[Y_n^2]= 1 = \E[Z_n^2]$, and \eqref{cov:Delta:u} implies that
\[
    \E[Y_nY_m] \le \sqrt{1-C_0} = \E[Z_nZ_m] \quad \forall m>n \ge n_0.
\]
By Slepian's lemma (Lemma \ref{lem:slepian}), for any fixed number $c>0$ and $M>m\ge n_0$,
\begin{align}
    \P\left( \bigcup_{n=m}^M \left\{ Y_n > c \sqrt{\ell(1/h_n)}\right\} \right) \ge \P\left( \bigcup_{n=m}^M \left\{ Z_n > c \sqrt{\ell(1/h_n)} \right\}\right),
\end{align}
where
\[
    \ell(x) := \log\log\log(x).
\]
Then, as in the proof of Theorem \ref{th:lm:aniso}, it follows that for any fixed $a>0$, $\delta>0$ and $M>m\ge n_0$,
\begin{align*}
    &\P(F_{m,M}) := \P\left\{ \sup_{x \in J: h_M\le |x-x_0| \le h_m} \frac{|u(t_0\,,x)-u(t_0\,,x_0)|}{\varphi_2(|x-x_0|) \sqrt{\ell(1/|x-x_0|)}} \ge a \right\}\\
    &\ge \P\left\{ \max_{m\le n \le M} \frac{u(t_0\,,x_n)-u(t_0\,,x_0)}{\|u(t_0\,,x_n)-u(t_0\,,x_0)\|_2\sqrt{\ell(1/h_n)}}\ge \frac{a}{C_1} \right\}\\
    & \ge \P\left( \bigcup_{n=m}^M \left\{ Y_n > \frac{a}{C_1} \sqrt{\ell(1/h_n)}\right\} \right) \ge \P\left( \bigcup_{n=m}^M \left\{ Z_n > \frac{a}{C_1} \sqrt{\ell(1/h_n)} \right\}\right)\\
    & \ge \P\left( \bigcup_{n=m}^M \left\{ \xi_n > \frac{(1+\delta)a \sqrt{\ell(1/h_n)}}{C_1(1-\sqrt{1-C_0})^{1/2}} \right\} \right) - \P\left\{ \xi_0 < - \frac{\delta a \sqrt{\ell(1/h_m)}}{C_1(1-C_0)^{1/4}}\right\}\\
    & \ge 1- \exp\left( - C\sum_{n=m}^M \frac{n^{-\frac{\gamma (1+\delta)^2a^2}{2C_1^2(1-\sqrt{1-C_0})}}}{\sqrt{\log(n^\gamma)}} \right) - \P\left\{ \xi_0 < - \frac{\delta a \sqrt{\log(m^\gamma)}}{C_1(1-C_0)^{1/4}}\right\}.
\end{align*}
Hence, for any fixed $\delta >0$ and $a>0$ with $a<C_1[2\gamma^{-1}(1-\sqrt{1-C_0})]^{1/2}/(1+\delta)$, we may let $M\to\infty$ first, then $m\to \infty$ to deduce that
\[
    \lim_{m\to\infty} \sup_{x \in J: |x-x_0|\le h_m} \frac{|u(t_0\,,x)-u(t_0\,,x_0)|}{\varphi_2(|x-x_0|)\sqrt{\ell(1/|x-x_0|)}} \ge a \quad \text{a.s.}
\]
Hence, we conclude that $K_6 \ge C_1[2\gamma^{-1}(1-\sqrt{1-C_0})]^{1/2}/(1+\delta)> 0$. 

Next, we prove that $K_6<\infty$.
For any $z \ge 0$, define
\[
    \mathscr{D}(z):=\Re \iint_{\sqrt{|\tau|^{2/\alpha}+|\xi|^2} \le e^z} (-i\xi) \hat{g}_{t_0,x_0}(\tau\,,\xi) |\tau|^{(1-2H)/2} \tilde{W}(\d \tau \, \d \xi),
\]
where $\hat{g}_{t,x}$ is as defined in \eqref{g:hat} and  $\tilde{W}=\tilde{W}_1 + i\tilde{W}_2$ as before.
Note that $\mathscr{D}(z)$ is a truncated version of the formal derivative of $v$ introduced in \eqref{v}.
Our idea is to linearize the spatial increment $u(t_0\,,x_0+h) - u(t_0\,,x_0)$ and approximate it by $\mathscr{D}(z) h$.
For any $n \in \N$ and $h \in [-1\,,1]$, define the approximation error
\[
    \mathscr{E}_n(h) := v(t_0\,,x_0+h) - v(t_0\,,x_0) - \mathscr{D}(n) h.
\]
\begin{lemma}
    Under \eqref{alphaH=3/2},
    \begin{align}\label{D:bd}
        &\|\mathscr{D}(z)-\mathscr{D}(z')\|_2 \lesssim \sqrt{|z-z'|}  \quad \text{and} \quad 
        \|\mathscr{D}(z)\|_2 \lesssim \sqrt{z}
    \end{align}
    uniformly for all $z,z'\ge 1$, and
    \begin{align}\label{E:bd}
        &\|\mathscr{E}_n(h)-\mathscr{E}_n(h')\|_2 \lesssim \varphi_2(|h-h'|) \quad \text{and} \quad 
        \|\mathscr{E}_n(h)\|_2 \lesssim e^{-n}
    \end{align}
    uniformly for all $n \ge 1$ and $|h|, |h'| \le e^{-n}$.
\end{lemma}

\begin{proof}
    For $1\le z \le z'$, by the change of variables $s = \tau^{1/\alpha}$, $z = (s\,,\xi)$, followed by polar coordinates $|z| = r$, and $2\alpha H = 3$,
    \begin{align*}
        &\|\mathscr{D}(z)-\mathscr{D}(z')\|_2^2\\
        &= \iint_{e^z \le \sqrt{|\tau|^{2/\alpha}+|\xi|^2}\le e^{z'}}  \frac{|\xi|^2|e^{-i\tau t_0}-e^{-t_0|\xi|^\alpha}|^2}{|\tau|^2 + |\xi|^{2\alpha}} |\tau|^{1-2H} \d \tau \, \d \xi\\
        & \lesssim \iint_{e^z \le \sqrt{|s|^{2}+|\xi|^2}\le e^{z'}}  \frac{|\xi|^2}{|s|^{2\alpha} + |\xi|^{2\alpha}} |s|^{\alpha(1-2H)} s^{\alpha-1}\d s \, \d \xi\\
        & \lesssim \iint_{e^z \le |z| \le e^{z'}} \frac{|z|^2}{|z|^{2\alpha}} |z|^{2\alpha-4}  \d z \lesssim \int_{e^z}^{e^{z'}} r^{-1} \d r = z'-z.
    \end{align*}
    Also, for all $z \ge 1$, we have
    \begin{align*}
        \|\mathscr{D}(z)\|_2^2
        & = \|\mathscr{D}(z)-\mathscr{D}(1)\|_2^2 + \iint_{ \sqrt{|\tau|^{2/\alpha}+|\xi|^2}\le e}  \frac{|\xi|^2|e^{-i\tau t_0}-e^{-t_0|\xi|^\alpha}|^2}{|\tau|^2 + |\xi|^{2\alpha}} |\tau|^{1-2H} \d \tau \, \d \xi\\
        & \lesssim (z-1) + \int_0^e \d \tau \int_0^e \d \xi \, |\xi|^2 |\tau|^{1-2H} \lesssim z.
    \end{align*}
    This proves \eqref{D:bd}.
    Next, for $n \ge 1$ and $|h|,|h'|\le e^{-n}$, by \eqref{u-u:x}, \eqref{D:bd} and $\log(1/|h-h'|) \ge \log e^n = n$, we have
    \begin{align*}
        \|\mathscr{E}_n(h)-\mathscr{E}_n(h')\|_2
        & \le \|v(t_0\,,x_0+h) - v(t_0\,,x_0+h')\|_2 + \|\mathscr{D}(n)\|_2 |h-h'|\\
        & \lesssim |h-h'| \sqrt{\log(1/|h-h'|)} + \sqrt{n}\, |h-h'|
        \le 2 \varphi_2(|h-h'|).
    \end{align*}
    Finally, by Taylor expansion, for $n\ge 1$, $|h|,|h'|\le e^{-n}$, we have
    \begin{align*}
        &\|\mathscr{E}_n(h)\|_2^2\\
        & \lesssim  \int_{\sqrt{|\tau|^{2/\alpha}+|\xi|^2} \le e^n} \left|e^{-i\xi(x_0+h)} - e^{-i\xi x_0} - (-i\xi) h e^{-i\xi x_0}\right|^2 \frac{|e^{-i\tau t_0}-e^{-t_0|\xi|^\alpha}|^2}{|\tau|^2+|\xi|^{2\alpha}} |\tau|^{1-2H} \d \tau \, \d \xi\\
        & \quad + \int_{\sqrt{|\tau|^{2/\alpha}+|\xi|^2} \ge e^n} \left|e^{-i\xi(x_0+h)}-e^{-i\xi x_0}\right|^2 \frac{|e^{-i\tau t_0}-e^{-t_0|\xi|^\alpha}|^2}{|\tau|^2+|\xi|^{2\alpha}} |\tau|^{1-2H} \d \tau \, \d \xi\\
        & \lesssim \int_{\sqrt{|\tau|^{2/\alpha}+|\xi|^2} \le e^n} |h|^4 |\xi|^4 \frac{|\tau|^{1-2H}}{|\tau|^2+|\xi|^{2\alpha}}  \d \tau \, \d \xi + \int_{\sqrt{|\tau|^{2/\alpha}+|\xi|^2} \ge e^n}  \frac{|\tau|^{1-2H}}{|\tau|^2+|\xi|^{2\alpha}}  \d \tau \, \d \xi\\
        & \lesssim |h|^4 \int_0^{e^{\alpha n}} \d \tau \int_0^{e^n} \d \xi \, |\xi|^{4-2\alpha} |\tau|^{1-2H}\\
        & \quad + \int_{\sqrt{|\tau|^{2/\alpha}+|\xi|^2} \ge e^n} \left[{\bf 1}_{\{|\tau|^{1/\alpha}\le |\xi|} + {\bf 1}_{\{|\tau|^{1/\alpha}>|\xi|\}}\right] \frac{|\tau|^{1-2H}}{|\tau|^2+|\xi|^{2\alpha}}  \d \tau \, \d \xi\\
        & \lesssim |h|^4 e^{2n} + \int_{e^n}^\infty \d \xi \int_0^{|\xi|^\alpha} \d \tau \, |\tau|^{1-2H} |\xi|^{-2\alpha} + \int_{e^{\alpha n}}^\infty \d \tau \int_0^{|\tau|^{1/\alpha}} \d \xi \, |\tau|^{-1-2H}\\
        & \lesssim e^{-2n}.
    \end{align*}
    This completes the proof of \eqref{E:bd}.
\end{proof}

\begin{lemma}\label{lem:lil:D:E}
    Under \eqref{alphaH=3/2}, there exists a constant $C\in (0\,,\infty)$ such that
    \begin{gather}\label{lil:D}
        \limsup_{n\to\infty} \frac{|\mathscr{D}(n)|}{\sqrt{n \log\log n}} \le C \quad \text{a.s. and}\\
        \limsup_{n\to\infty} \sup_{h \in [e^{-n-1},e^{-n}]} \frac{|\mathscr{E}_n(h)|}{h \sqrt{\log(1/h) \log\log\log(1/h)}} = 0 \quad \text{a.s.}\label{lil:E}
    \end{gather}
\end{lemma}

\begin{proof}
    By \eqref{D:bd} and Dudley's theorem \cite{Dudley} (see also \cite[Theorem 1.3.3]{AT} or \cite[Theorem 6.1.2]{MR}), there exist constants $c_1, c_2>0$ such that for all large $k$,
    \begin{align*}
        \E\sup_{z \in [e^{k-1},e^k]}|\mathscr{D}(z)| \lesssim \int_0^{c_1 e^{k/2}} \sqrt{\log \left(\frac{c_2 e^k}{r^2}\right)}\, \d r \lesssim e^{k/2}.
    \end{align*}
    Let $C>0$.
    The preceding and Borell's inequality \cite{Borell} (see also \cite[Theorem 2.1.1]{AT}) imply that for $k$ large,
    \begin{align*}
        &\P\left\{ \sup_{z \in [e^{k-1},e^k]}|\mathscr{D}(z)| > C \sqrt{e^k \log\log e^k} \right\}\\
        &\le \P\left\{ \sup_{z \in [e^{k-1},e^k]}|\mathscr{D}(z)| - \E\sup_{z \in [e^{k-1},e^k]}|\mathscr{D}(z)| > \tfrac12 C \sqrt{e^k \log k} \right\}\\
        &\le 2 \exp\left( - \frac{4^{-1} C^2 e^k \log k}{2 \sup_{z \in [e^{k-1},e^k]}\|\mathscr{D}(z)\|_2^2} \right),
    \end{align*}
    which is summable for a large enough constant $C$, thanks to \eqref{D:bd}.
    An application of the Borel-Cantelli lemma and a simple monotonicity argument then yield \eqref{lil:D}.

    Next, we prove \eqref{lil:E}. 
    To this end, we apply Dudley's theorem \cite{Dudley}, \eqref{E:bd}, and $\varphi^{-1}_2(r) \asymp r (\log_+ (1/r))^{-1/2}$ as $r\to0^+$ to deduce that there exist $c_3,c_4>0$ such that for all large $n$,
    \begin{align*}
        &\E\sup_{h \in [e^{-n-1},e^{-n}]} |\mathscr{E}_n(h)| \lesssim \int_0^{c_3 e^{-n}} \sqrt{\log\bigg( \frac{c_4 e^{-n}}{r/\sqrt{\log(1/r)}} \bigg)} \, \d r\\
        & \lesssim \int_0^{c_3e^{-n}} \left[\sqrt{\log(c_4 e^{-n}/r)} + \sqrt{\log\log(1/r)} \right] \d r\\
        & \lesssim e^{-n} + e^{-n} \sqrt{\log\log e^n} + \int_0^{c_3 e^{-n}} \frac{\d r}{\log(1/r)\sqrt{\log\log(1/r)}}\\
        & \lesssim e^{-n} + e^{-n}\sqrt{\log n} + e^{-n}\\
        &\lesssim e^{-n}\sqrt{\log n},
    \end{align*}
    where we have used $\sqrt{a+b}\le \sqrt{a}+\sqrt{b}$ for $a,b\ge 0$ in the second inequality and integration by parts in the third inequality.
    Let $\varepsilon >0$.
    Then, thanks to the preceding, Borell's inequality \cite{Borell}  and \eqref{E:bd}, there exists $c_5>0$ such that for all large $n$,
    \begin{align*}
        &\P\left\{ \sup_{h \in [e^{-n-1},e^{-n}]} |\mathscr{E}_n(h)| > \varepsilon e^{-n}\sqrt{n \log\log n} \right\}\\
        &\le P\left\{ \sup_{h \in [e^{-n-1},e^{-n}]} |\mathscr{E}_n(h)| - \E \sup_{h \in [e^{-n-1},e^{-n}]} |\mathscr{E}_n(h)| > \tfrac12 \varepsilon e^{-n}\sqrt{n \log\log n} \right\}\\
        & \le 2 \exp\left( - \frac{4^{-1}\varepsilon^2 e^{-2n} n \log\log n}{2 \sup_{h \in [e^{-n-1},e^{-n}]}\|\mathscr{E}_n(h)\|_2^2} \right) \le 2 \exp\left( -\frac{\varepsilon^2 n \log\log n}{c_5}  \right).
    \end{align*}
    Hence, the Borel-Cantelli lemma, a monotonicity argument, and the arbitrariness of $\varepsilon >0$ imply \eqref{lil:E}.
\end{proof}

Thanks to Lemma \ref{lem:lil:D:E}, we obtain the following:
\begin{align*}
    &\lim_{n\to \infty}\sup_{h \in [e^{-n-1},e^{-n}]} \frac{|u(t_0\,,x_0+h)-u(t_0\,,x_0)|}{h \sqrt{\log(1/h) \log\log\log(1/h)}}\\
    & \le \lim_{n\to \infty}\sup_{h \in [e^{-n-1},e^{-n}]} \frac{|\mathscr{D}(n)h|+|\mathscr{E}_n(h)|}{h \sqrt{\log(1/h) \log\log\log(1/h)}}\\
    & \le \limsup_{n \to \infty} \frac{|\mathscr{D}(n)|}{\sqrt{n \log\log n}} + 0
    \le C \quad \text{a.s.}
\end{align*}
Hence, $K_6 \le C < \infty$. This proves \eqref{lm:SPDE:x}.

Next, we prove \eqref{lm:SPDE:t}.
Let $I \subset (0\,,\infty)$ be a compact interval containing the fixed point $t_0>0$.
Lemma \ref{lem:SPDE:reg} implies that
\[
    \|u(t\,,x_0)-u(s\,,x_0)\|_2 \asymp \varphi_1(|t-s|) := |t-s|^{1/\alpha} \quad \forall t,s \in I,
\]
and $\varphi_1$ satisfies \eqref{lm:phi}.
Also, since $\{u(t\,,x_0)\}_{t \in I}$ satisfies SLND \cite[Theorem 1.3]{CLX}, condition \eqref{lm:cv} holds for $|t_n-t_0| = 2^{-n}$ (see Proposition \ref{prop:SLND}).
Therefore, Theorem \ref{th:lm} implies that \eqref{lm:SPDE:t} holds with $0<K_7<\infty$.

It remains to prove \eqref{lm:SPDE:tx}.
By Lemma \ref{lem:SPDE:reg}, Assumption \ref{a:sigma:aniso} is satisfied.
By \eqref{SPDE:psi} and \eqref{SPDE:psi:inv},
\begin{align*}
    \int_0^1 \frac{\psi^{-1}(hr)\, \d r}{r\sqrt{\log(2/r)}}
    &\lesssim  \psi^{-1}(h) \int_0^1 \frac{r^{1/(\alpha+1)}\big(\frac{\log_+(1/(hr))}{\log_+(1/h)}\big)^{1/(2\alpha+2)}\, \d r}{r\sqrt{\log(2/r)}}\\
    & \lesssim \psi^{-1}(h) \int_0^1 \frac{r^{1/(\alpha+1)}\, \d r}{r\sqrt{\log(2/r)}} \quad \text{as $h\to0^+$}
\end{align*}
by the dominated convergence theorem.
Since the last integral is finite, Assumption \ref{a:psi:lil} is satisfied with $\ell(x) = \log\log(x)$.
Hence, Theorem \ref{th:lm:aniso} implies that \eqref{lm:SPDE:tx} holds with $K_8<\infty$.
Finally, by \eqref{lm:SPDE:t},
\begin{align*}
    &\lim_{h\to0^+} \sup_{z: 0<\phi(z, z_0) \le h} \frac{|u(z)-u(z_0)|}{\phi(z\,,z_0)\sqrt{\log\log(1/\phi(z\,,z_0))}}\\
    & \ge \lim_{h\to0^+} \sup_{t: 0<|t-t_0|^{1/\alpha} \le h} \frac{|u(t\,,x_0)-u(t_0\,,x_0)|}{|t-t_0|^{1/\alpha}\sqrt{\log\log(1/|t-t_0|^{1/\alpha})}} = K_7>0.
\end{align*}
Hence $K_8>0$.
\qed

\section{Gaussian Volterra processes}
\label{s:volterra}

In this section, we discuss another example about a class of Gaussian Volterra processes.
Fix $T>0$, and define
\[
    X(t) = \int_0^t \kappa(t-r)\, \d B(r) \quad \forall t \in [0\,,T],
\]
where $\{B(t)\}_{t \ge 0}$ is a standard Brownian motion. 
Mocioalca and Viens \cite{MV05} have studied, among other things, continuity properties and moduli of continuity of $X = \{X(t)\}_{t \in [0,T]}$.
We consider a setting similar to \cite{MV05}.

\begin{assumption}\label{a:kappa}
    $\kappa(r) = \sqrt{(\phi^2)'(r)}$ where $\phi: [0\,,T] \to (0\,,\infty)$ is non-decreasing, continuous, $C^1$ except perhaps at 0, $\phi(0)=0$ and $(\phi^2)' \ge 0$ is non-increasing.
\end{assumption}

\begin{lemma}\label{lem:volterra}
Under Assumption \ref{a:kappa}, we have:
\begin{enumerate}[$(i)$]
    \item $\phi(|t-s|) \le \|X(t)-X(s)\|_2 \le \sqrt{2}\, \phi(|t-s|)$ uniformly for all $t,s \in [0\,,T]$; and\smallskip
    \item One-sided SLND:
    \[
        \Var(X(t)|X(t_1),\dots, X(t_n)) \ge \phi^2(t-t_n) \ge \tfrac12 \Var(X(t)-X(t_n))
    \]
    uniformly for all $n \in \N$ and $0 \le t_1 < \cdots < t_n \le t \le T$.
\end{enumerate}
\end{lemma}

\begin{proof}
First, (i) follows from the same proof of Proposition 1 in \cite{MV05}. Next, for any $n \in \N$ and $0 \le t_1 < \cdots < t_n \le t$, we have
\begin{align*}
    &\Var(X(t)|X(t_1),\dots, X(t_n))
    = \inf_{a_1,\dots, a_n\in \R} \E\left[ \left( X(t) - \sum_{i=1}^n a_i X(t_i) \right)^2 \right]\\
    &= \inf_{a_1,\dots, a_n\in \R} \left[ \int_{t_n}^t \kappa^2(t-r)\, \d r + \int_0^{t_n} \left( \kappa(t-r) - \sum_{i=1}^n a_i \kappa(t_i-r) {\bf 1}_{[0,t_i]}(r) \right)^2 \, \d r \right]\\
    & \ge \int_{t_n}^t \kappa^2(t-r)\, \d r = \phi^2(t-t_n) \ge \tfrac12  \Var(X(t)-X(t_n)),
\end{align*}
where the last inequality follows from (i). This proves (ii).
\end{proof}

Thanks to Lemma \ref{lem:volterra}, if $\phi$ is regularly varying of index $\theta\in (0\,,1]$, then Theorems \ref{th:um} and \ref{th:lm} can be applied to obtain exact uniform and local moduli of continuity for $X$ (see Remark \ref{rmk:theta>0} and Proposition \ref{prop:SLND}). 
We refer to \cite{MV05, H23} for related results.

It follows from \eqref{Fernique} that if $\phi(r) \gtrsim (\log(1/r))^{-1/2}$, then $X$ is a.s.~not continuous.
In \cite{MV05}, the process $X$ with
\[
    \kappa^2(r) = (\phi^2)'(r) = \beta r^{-1}(\log(1/r))^{-\beta-1}, \quad \phi^2(r) = (\log(1/r))^{-\beta} \quad (\beta>1)
\]
is studied and referred to as the logarithmic Brownian motion with parameter $\beta$.
Next, we discuss the exact moduli of continuity of $X$ in a similar regime where $\phi$ is slowly varying and $X$ is a.s.~continuous but not H\"older continuous.

\begin{assumption}\label{a:sv}
The function $\phi$ is slowly varying at 0 and takes the form
\begin{align*}
    \phi(r) = \frac{L(\log(1/r))}{(\log(1/r))^a}
\end{align*}
where $L$ is slowly varying at infinity.
\end{assumption}

\begin{example}
Assumption \ref{a:kappa} holds on $[0\,,T]$ for some $T>0$ if $L$ is one of the following: $(i)$ $L\equiv 1$; $(ii)$ $L(x) = \log x$; or $(iii)$ $L(x) = \exp(\sqrt{\log x})$.
\end{example}

\begin{lemma}\label{lem:vol:cond}
    If $\phi$ satisfies Assumptions \ref{a:kappa} and \ref{a:sv} with $a>1/2$, then
    \begin{enumerate}[$(i)$]
        \item $\int_0^h \frac{\phi(r) \, \d r}{r\sqrt{\log(1/r)}} \lesssim \phi(h) \sqrt{\log(1/h)} = o(1)$ as $h \to 0^+$;
        \item $\int_0^1 \frac{\phi(hr) \, \d r}{r \sqrt{\log(2/r)}} \lesssim \phi(h) \sqrt{\log(1/h)} = o(1)$ as $h \to 0^+$;
        \item For any $\gamma > 1$, there exist $C_\gamma>0$, and $n_\gamma \ge 1$ such that
        \begin{align*}
            \Var(X(t_n)-X(t_0)|X(t_m)-X(t_0)) \ge C_\gamma \Var(X(t_n)-X(t_0))
        \end{align*}
        uniformly for all $m > n \ge n_\gamma$, where $t_n = t_0 + n^{-\gamma}$.
    \end{enumerate}
\end{lemma}

\begin{proof}
First, by the change of variable $r=e^{-x}$ and \cite[Proposition 1.5.10]{BGT} (since $a+1/2>1$),
\begin{align*}
    \int_0^h \frac{\phi(r)\, \d r}{r\sqrt{\log(1/r)}} = \int_{\log(1/h)}^\infty \frac{L(x) \, \d x}{x^{a+1/2}} \lesssim \frac{L(\log(1/h))}{(\log(1/h))^{a-1/2}} \quad \text{as $h \to 0^+$.}
\end{align*}
The last expression is equal to $\phi(h)\sqrt{\log(1/h)}$, which is $o(1)$ since $a>1/2$.
This proves (i).
Next, by monotonicity of $\phi$ and (i), as $h\to0^+$,
\begin{align*}
    \int_0^1 \frac{\phi(hr) \, \d r}{r \sqrt{\log(2/r)}}
    &\le \int_0^h \frac{\phi(r)\, \d r}{r\sqrt{\log(2/r)}} + \int_h^1 \frac{\phi(h)\, \d r}{r \sqrt{\log(2/r)}} \lesssim \phi(h)\sqrt{\log(1/h)},
\end{align*}
which yields (ii). Finally, for any $\gamma>1$, we have
\begin{align*}
    t_n-t_{n+1} &= n^{-\gamma} - (n+1)^{-\gamma} = n^{-\gamma} (n+1)^{-\gamma} ((n+1)^\gamma-n^\gamma)\\
    &\ge n^{-\gamma} (n+1)^{-\gamma} (\gamma n^{\gamma-1})
    = \gamma n^{-1} (n+1)^{-\gamma},
\end{align*}
where the inequality follows from the mean value theorem. So, there is $\nu > 1+1/\gamma$ such that $t_n-t_{n+1} \ge (t_n-t_0)^\nu$ for all large $n$.
Then, by one-sided SLND (Lemma \ref{lem:volterra}(ii)), monotonicity of $\phi$ and the slowly varying property of $L$, we can find $C_\gamma>0$ and $n_\gamma \ge 1$ such that
\begin{align*}
    &\Var(X(t_n)-X(t_0)|X(t_m)-X(t_0))
    \ge \Var(X(t_n)|X(t_m),X(t_0))\\
    &\ge \phi^2(t_n-t_m)
    \ge \phi^2((t_n-t_0)^\nu)
    = \frac{L^2(\nu\log(1/(t_n-t_0)))}{(\nu\log(1/(t_n-t_0)))^{2a}}\\
    &\ge C_\gamma \frac{L^2(\log(1/(t_n-t_0)))}{(\log(1/(t_n-t_0)))^{2a}} = C_\gamma \phi^2(t_n-t_0) \ge \frac{C_\gamma}{2} \Var(X(t_n)-X(t_0))
\end{align*}
uniformly for all $m > n \ge n_\gamma$.
This completes the proof.
\end{proof}

Finally, we present the following result, which shows that the exact uniform and local modulus functions are the same.
A similar behavior is known for Gaussian processes with stationary increments where $\sigma(t\,,s) = \phi(|t-s|) = (\log(1/|t-s|))^{-a}$ and $a>1/2$ (see \cite{M68}).

\begin{theorem}
    Fix $T>0$ and $t_0 \in [0\,,T]$.
    If $\phi$ satisfies Assumptions \ref{a:kappa} and \ref{a:sv} with $a>1/2$, then there is a constant $C \in (0\,,\infty)$ such that
    \begin{align}\label{um:volterra}
        \lim_{h\to0^+} \sup_{t,s \in [0,T]: 0<|t-s|\le h} \frac{|X(t)-X(s)|}{\phi(|t-s|)\sqrt{\log(1/|t-s|)}} = C \quad \text{a.s.}
    \end{align}
    and there is a non-random number $K=K(t_0) \in (0\,,\infty)$ such that
    \begin{align}\label{lm:volterra}
        \lim_{h\to0^+}\sup_{t\ge 0: 0<|t-t_0| \le h} \frac{|X(t)-X(t_0)|}{\phi(|t-t_0|)\sqrt{\log(1/|t-t_0|)}} = K \quad \text{a.s.}
    \end{align}
\end{theorem}

\begin{proof}
Thanks to Lemma \ref{lem:volterra} and Lemma \ref{lem:vol:cond}(i), we may apply Theorem \ref{th:um} to obtain \eqref{um:volterra}.
Next, Lemma \ref{lem:volterra}(i) and Lemma \ref{lem:vol:cond}(ii)--(iii) imply that Assumptions \ref{a:sigma:aniso}, \ref{a:psi:lil} and \ref{a:cv:lil} hold with $\ell(x) = \log(x)$ and $h_n = n^{-\gamma}$.
Hence, an application of Theorem \ref{th:lm:aniso} yields \eqref{lm:volterra}.
\end{proof}

{\bf Acknowledgment.}
C.Y. Lee was supported in part by the Shenzhen Peacock grant 2025TC0013.

\bibliography{ref}
\bibliographystyle{abbrv}

\end{document}